\documentclass[a4paper, 10pt, twoside, notitlepage]{amsart}

\usepackage[utf8]{inputenc} 
\usepackage{color}
\usepackage{amsmath} 
\usepackage{amssymb} 
\usepackage{amsthm}
\usepackage{geometry}
\usepackage{graphicx}
\usepackage{esint}
\usepackage[colorlinks=true,linkcolor=blue]{hyperref}
\usepackage{bbold}
\usepackage{xfrac}

\usepackage{enumitem}

\theoremstyle{plain}
\newtheorem{thm}{Theorem}
\newtheorem{prop}{Proposition}[section]
\newtheorem{lem}[prop]{Lemma}

\newtheorem{rmk}[prop]{Remark}

\newcommand {\R} {\mathbb{R}}

\newcommand {\N} {\mathbb{N}}
 
\newcommand {\Sn} {\mathbb{S}^n}
\newcommand {\p} {\partial}

\newcommand {\D} {\Delta}

\newcommand {\diam} {\text{diam}}

\newcommand{\al}{\alpha}

\newcommand{\half}{\text{\normalfont\sfrac{1}{2}}}

\DeclareMathOperator{\supp}{supp}

\DeclareMathOperator {\dist} {dist}

\DeclareMathOperator{\F} {\mathcal{F}}

\title[Runge Approximation and Stability Improvement]{Runge Approximation and Stability Improvement for a Partial Data Calder\'on problem for the Acoustic Helmholtz Equation}

\author{Mar\'ia \'Angeles Garc\'ia-Ferrero}
\address{Institut für Angewandte Mathematik, Ruprecht-Karls-Universität Heidelberg, Im Neuenheimer Feld 205, 69121 Heidelberg, Germany}
\email{garciaferrero@uni-heidelberg.de}
\author{Angkana Rüland}
\address{Institut für Angewandte Mathematik, Ruprecht-Karls-Universität Heidelberg, Im Neuenheimer Feld 205, 69121 Heidelberg, Germany}
\email{Angkana.Rueland@uni-heidelberg.de}
\author{Wiktoria Zato\'n}
\address{Institut für Mathematik, Universität Zürich, Winterthurerstrasse 190, 8057 Zürich, Switzerland}
\email{wiktoria.zaton@math.uzh.ch}

\begin{document}
\maketitle

\begin{abstract}
In this article, we discuss quantitative Runge approximation properties for the acoustic Helmholtz equation and prove stability improvement results in the high frequency limit for an associated partial data inverse problem modelled on \cite{AU04, KU19}. The results rely on quantitative unique continuation estimates in suitable function spaces with explicit frequency dependence. We contrast the frequency dependence of interior Runge approximation results from non-convex and convex sets.  
\end{abstract}

\section{Introduction}

In this article we study improvement of stability effects in Runge approximation originating from the interplay of geometry and an increasing frequency parameter for the  acoustic Helmholtz equation. These effects had first been observed in \cite{HI04} and have subsequently been the object of intensive study, both in the context of unique continuation \cite{IK11,I19,I07,IS10,IS07} and with regards to their effects on inverse problems \cite{CI16,I18,I11,EI18,EI20,IN14,IN14a,ILW26,I13,IW20,BLT10,
NUW13,BLZ20,INUW14}. Due to the notorious instability in many inverse problems, these improved stability estimates are of great significance, both from a theoretical and practical point of view \cite{BNO19}. We refer to Section ~\ref{sec:literature} for an (incomplete) overview of the history and background of these type of results.

Building on the observation that Runge approximation properties are qualitatively and quantitatively dual to unique continuation \cite{RS17a} (see also \cite{L88,Z07} and the references therein for analogous results in the control theory community), in this article we seek to study the effects of geometry and increasing frequency  $k$ for acoustic Helmholtz  equations
\begin{equation}
    \label{eq:generaleq}
    \big(\Delta+k^2q(x)+V(x)\big)u(x) = 0
\end{equation} 
on the associated Runge approximation properties under suitable conditions on the geometry and the potentials (see Section ~\ref{sec:setting}). For the special case of the (pure) Helmholtz equation ($q=1$ and $V=0$) quantitative Runge approximation had been deduced in \cite[Lemma 2.1]{EPS19} in the context of approximation properties for dispersive equations. Relying on the duality between Runge approximation and unique continuation, we here prove quantitative unique continuation properties for the acoustic Helmholtz equation \eqref{eq:generaleq} in suitable, adapted function spaces, carefully tracking the parameter dependence. 
 
 As two of our main results, we deduce Runge approximation properties with \emph{exponential} $k$-dependence \emph{without} geometric assumptions (Theorems ~\ref{thm:Rungeboundary} and ~\ref{thm:Rungeinterior}) and \emph{improved, polynomial} behaviour under \emph{convexity} conditions on the domain (Theorem ~\ref{thm:Rungeinterior_improv}). 
 
In order to illustrate the importance, robustness and usefulness of these estimates, we consider the partial data inverse problem for the equation \eqref{eq:generaleq}. Using the systematic duality strategy from \cite{RS17b}, we prove improved stability results for this nonlinear inverse problem (Proposition ~\ref{prop:stability}). This generalizes the results from \cite{KU19} to the case of acoustic Helmholtz equations.

\subsection{Setting}
\label{sec:setting}
In the following, we outline the precise geometric and functional assumptions under which our results are valid.  Here, as a model system, we focus on generalizations of Helmholtz type equations with homogeneous Dirichlet conditions under the assumption that the increasing parameter is chosen with some distance to the spectrum. More precisely, for $n\ge2$, $k\ge 1$ and  $\Omega\subseteq \R^n$ a bounded, connected, open set with Lipschitz boundary we consider the acoustic Helmholtz equation \eqref{eq:generaleq}
in $\Omega$, where   
\begin{enumerate}[leftmargin=*]
	\item[(i)] \label{assV}$V\in L^\infty(\Omega)$ and zero is not a Dirichlet-eigenvalue of $\D + V$ in $\Omega$,
	\item[(ii)] \label{assq}$q\in C^1(\Omega)$ and $q$ is strictly positive, that is $0<\kappa^{-1}<q<\kappa$ for some $\kappa>1$. 
\end{enumerate}  
Let us comment on these conditions: The assumption that $V\in L^{\infty}(\Omega)$ ensures that the potential $V$ is subcritical in terms of scaling and that it is well in the regime in which unique continuation results are available. The second condition in \textnormal{(\hyperref[assV]{i})} is a (technical) solvability condition. By domain perturbation arguments, this is generically satisfied \cite{Kato}.

The conditions formulated in \textnormal{(\hyperref[assq]{ii})} on $q$ are two-fold: The sign condition and bounds on $q$ ensure that the acoustic equation \eqref{eq:generaleq} is of Helmholtz type; $q=1$ corresponds to the Helmholtz equation with potential. The regularity condition $q\in C^1(\Omega)$ will be used in order to treat the term $k^2 q(x)$ as part of the principal symbol of the operator. In Sections ~\ref{sec:qUCP_without_geo} and ~\ref{sec:Runge} this will be a consequence of introducing an auxiliary new dimension in order to reduce the parameter dependent equation to a non-parameter dependent equation with $C^1$ principal symbol. In Section ~\ref{sec:improved_results_convexity} we will directly treat the parameter dependent term $k^2 q$ as part of the principal symbol of the Carleman estimate for which we again require some regularity on $q$. In order to do so, we will complement the condition \textnormal{(\hyperref[assq]{ii})} with an additional radial monotonicity assumption (see \textnormal{(\hyperref[assq2]{ii'})} in Section ~\ref{sec:improved_results_convexity_a}).

We consider \eqref{eq:generaleq} with homogeneous Dirichlet boundary conditions. In order to avoid solvability issues or a priori estimates without control on $k$, we make an additional technical assumption: We suppose that the real parameter $k\geq 1$ is chosen such that zero is not a Dirichlet eigenvalue of the operator in $\Omega$. More precisely,
let $\Sigma_{V,q}$ denote the set of the inverse of eigenvalues of the operator $T:=(-\D-V)^{-1} M_q$, where $M_q$ denotes the multiplication operator with $q$. We will then assume that
\begin{enumerate}[leftmargin=.75cm]
\item[(a1)] \label{assSpec} $\dist(k^2,\Sigma_{V,q})>ck^{2-n}$, for some $c\ll 1$.
\end{enumerate}
We remark that generically this does not pose major restrictions, as it is always possible to find arbitrarily large values of $k$ such that the condition \textnormal{(\hyperref[assSpec]{a1})} is fulfilled:
Indeed, the operator $T$ is a classical pseudodifferential operator of order $-2$. We denote  the spectrum of $T^{-1}$ by $\Sigma_{V,q}=\{\lambda_n\}_{n\in\N}$. By Weyl's law
\cite{HormIV, S87}
$$\#\{\lambda_n\leq E\}= C E^{\frac n 2} \mbox{ as } E\to\infty.$$
Thus, on average, the distance between consecutive eigenvalues is $C'E^{1-\frac n 2}$ as $E\to\infty$.
In this case, \textnormal{(\hyperref[assSpec]{a1})}  ensures that it is possible to find admissible frequencies in essentially all frequency ranges with $c\in(0,\frac {C'}{3})$.

We remark that also other boundary conditions would have been feasible. An alternative, natural condition would have been impedance conditions (including the potential $q$) which in the limit of growing domains would have approximated a Sommerfeld type radiation condition. This would also have had the advantage of avoiding the eigenvalue assumptions and discussions. Since we are working in finite domains, for simplicity, we here restrict our attention to the Dirichlet setting.

\subsection{Runge approximation without convexity conditions}

With the conditions stated above, we first address Runge approximation results \emph{without} additional convexity assumptions on the domain. Our main results then provide quantitative Runge approximation results with a quantified dependence on the parameter $k$. Since these properties are dual to unique continuation properties for which \emph{exponential} dependences on $k$ are unavoidable without additional geometric assumptions  \cite{BM20}, the dependences on $k$ are expected to be exponential.

For the case of approximation in the domain in which a solution is prescribed we thus obtain the following result:

\begin{thm}\label{thm:Rungeboundary}
Let $\Omega_1, \Omega_2 \subset\R^n$ be  open, bounded, connected Lipschitz domains such that $\Omega_1\Subset \Omega_2$ and such that $\Omega_2\backslash\overline\Omega_1$ is connected. Let $\Gamma$ be a non-empty,  open subset of $\p \Omega_2$.
Let $V$ and $q$ satisfy \textnormal{(\hyperref[assV]{i})}-\textnormal{(\hyperref[assq]{ii})} in $\Omega_2$.
There exist  constants $\mu>1$, $s\geq n+6$ and $C>1$ depending on $n, \Omega_2, \Omega_1,  \|V\|_{L^\infty(\Omega_2)}, \kappa$ and $\|q\|_{C^1(\Omega_2)}$ such that for any
solution $v \in H^1(\Omega_1)$ of
\begin{align*}
(\Delta+k^2q+V)v = 0 \quad \mbox{in} \quad\Omega_1,
\end{align*}
with $k\geq 1$ satisfying \textnormal{(\hyperref[assSpec]{a1})},
 and any $\epsilon>0$, there exists a solution $u$ to
\begin{align*}
(\Delta+k^2q+V)u = 0 \quad \mbox{in} \quad\Omega_2,
\end{align*}
with  $ u|_{\p\Omega_2}\in \widetilde{H}^{\half}(\Gamma)$ 
such that 
\begin{align}\label{eq:Rungeboundary2}
\|u-v\|_{L^2(\Omega_1)} \leq \epsilon \|v\|_{H^1(\Omega_1)}, \qquad \|u\|_{H^{\half}(\partial \Omega_2)} \leq C e^{Ck^s \epsilon^{-\mu}} \|v\|_{L^2(\Omega_1)}.
\end{align}
\end{thm}

In Section ~\ref{sec:optimality} we show that in the full data case and for $q=1$, up to the precise values of $\mu$, $s$ and $C$, the bound in $\epsilon$ is optimal, see \cite[Section 5]{RS17a} for the analogous result for the Laplacian without a large parameter.

If $v$ is a solution in a slightly larger domain than the one for which we seek to find a good approximation, the exponential dependence in $\epsilon$ changes to a polynomial dependence while the $k$-dependence remains exponential:

\begin{thm}\label{thm:Rungeinterior}
Let $\Omega_1, \Omega_2, \Gamma, V$ and $q$ be as in Theorem ~\ref{thm:Rungeboundary}. Further, let $\tilde\Omega_1$ be a 
bounded, Lipschitz domain such that $\Omega_1 \Subset \tilde\Omega_1\Subset\Omega_2$. There exist constants $\nu>1$ and $C>1$ depending on $n, \Omega_2, \Omega_1,  \|V\|_{L^\infty(\Omega_2)}, \kappa$ and $\|q\|_{C^1(\Omega_2)}$ such that for any solution $\tilde v \in H^1(\tilde \Omega_1)$ of \begin{align*}
(\Delta+k^2q+V) \tilde v = 0\; \mbox{ in } \tilde\Omega_1,
\end{align*}
with $k\geq 1$ satisfying \textnormal{(\hyperref[assSpec]{a1})}, there exists a solution $u$ to
\begin{align*}
(\Delta+k^2q+V)u = 0 \; \mbox{ in } \Omega_2
\end{align*}
with $u|_{\p\Omega_2}\in \widetilde{H}^{\half}(\Gamma)$
such that 
\begin{align}\label{eq:Rungeinterior2}
\|u-\tilde v\|_{L^2(\Omega_1)} \leq \epsilon \|\tilde v\|_{H^1(\tilde\Omega_1)}, \qquad \|u\|_{H^{\half}(\Gamma)}\leq  C{ e^{Ck}}\epsilon^{-\nu}\|\tilde v\|_{L^2( \Omega_1)}.
\end{align}
\end{thm}

As an application of these results we prove a partial data uniqueness result for the Calder\'on problem with stability improvement in $k$ under a priori assumptions on the potential in a neighbourhood of the boundary. This generalizes the results from \cite{KU19} to accoustic equations. 
In particular, it thus combines ideas from \cite{AU04,RS17} with the observations from \cite{HI04} (see also the references above and below). We further refer to \cite{I11} for similar results for different ranges of $k$.

\begin{prop}\label{prop:stability}
Let $n\geq 3$, let $\Omega\subset\R^n$ be a bounded, connected, smooth open set and let $\Gamma\subset\p\Omega$ be a nonempty open subset. Let $\Omega'\Subset\Omega$ be an open, Lipschitz subset such that  $\Omega\backslash\Omega'$ is connected.
Let  $q_1, q_2, V_1, V_2$ verify \textnormal{(\hyperref[assV]{i})}-\textnormal{(\hyperref[assq]{ii})} in $\Omega$ and be such that
\begin{align*}
\|q_j\|_{L^\infty(\Omega)}+\|V_j\|_{L^\infty(\Omega)}\leq B,\\
q_1=q_2,\; V_1=V_2 \mbox{ in } \Omega\backslash\Omega'.
\end{align*} 
Then, there exists a constant $C>1$ depending on $n, \Omega, \Omega',\Gamma $, $\|q_j\|_{C^1(\Omega)}$ and  $B$ such  that  for all $k\geq 1$ such that
$\dist(k^2, \Sigma_{V_j,q_j}) > ck^{2-n}$
and 
$\delta=\|\Lambda_{V_1,q_1}^\Gamma(k)-\Lambda_{V_2,q_2}^\Gamma(k)\|_{\tilde H^{\half}(\Gamma)\to H^{-\half}(\Gamma)}<1$, we have
\begin{align*}
\|k^2(q_2-q_1)+(V_2-V_1)\|_{H^{-1}(\Omega)}
\leq C\Bigg(e^{Ck^{n+3}} \delta+\frac{1}{\left(k+|\log\delta|^{\frac{1}{n+3}}\right)^{\frac 2 n}}\Bigg).
\end{align*}
\end{prop}

\begin{rmk}
\label{rmk:two_meas}
We remark that if in Proposition ~\ref{prop:stability} two measurements for different values of $k$ are available, it is possible to provide stability both for $q_j$ and $V_j$ separately.
\end{rmk}

Earlier improvement of stability results had been obtained for the corresponding \emph{full data} inverse problems in \cite{NUW13, INUW14}. While in the case of the Helmholtz equation ($q=1$) with potential,
the $k$-dependence of the Lipschitz contribution was proved to be \emph{polynomial} instead of exponential, already in the full data acoustic case ($q\neq 1$) exponential $k$-dependences emerged.

\subsection{Improvements of the Runge approximation results in convex geometries}

\label{sec:improved_results_convexity_a}

Last but not least, in our final section,  in line with the observations from \cite{HI04} (and the literature building up on this), showing that in convex domain geometries  the $k$-dependences in quantitative unique continuation improve with a large parameter, we also obtain improved Runge approximation results in the interior in the presence of a large parameter.
 Here we impose an additional monotonicity condition on the potential $q$ which is well-known in the context of the study of embedded eigenvalues \cite{KT06}:

\begin{itemize}
\item[(ii')] \label{assq2}
$q\in C^1(\Omega)$, $\kappa^{-1}\leq q\leq \kappa$ for some $\kappa>1$ and $\nabla q\cdot x\geq 0$.
\end{itemize}

With this assumption we deduce improved (in $k$) dependences in the Runge approximation results in the interior for convex geometries:

\begin{thm}\label{thm:Rungeinterior_improv}
Let $V$ and $q$ be as in \textnormal{(\hyperref[assV]{i})}-\textnormal{(\hyperref[assq2]{ii'})} in $\Omega_2=B_2 \backslash \overline{B_{\half}}$ and let $\Omega_1=B_1 \backslash \overline{B_{\half}}$ and $\tilde \Omega_1=B_{1+\delta}\backslash \overline{B_{\half}}$, for some $\delta\in(0,1)$.  There exist  parameters $\nu>1$, $s>3$ and a constant $C>1$ depending on $n,  \|V\|_{L^\infty(\Omega_2)}, \kappa$ and $\|q\|_{C^1(\Omega_2)}$ such that for any solution $\tilde v \in H^1(\tilde \Omega_1)$ of \begin{align*}
(\Delta+k^2q+V) \tilde v &= 0 \; \mbox{ in }  \tilde \Omega_1,\\
\tilde v&=0 \; \mbox{ on }  \p B_{\half},
\end{align*}
with $k\geq 1$ satisfying \textnormal{(\hyperref[assSpec]{a1})}, there exists a solution $u$ to
\begin{align*}
(\Delta+k^2q+V)u &= 0 \; \mbox{ in } \Omega_2,\\u&=0 \; \mbox{ on }  \p B_{\half},
\end{align*}
such that 
\begin{align*}
\|u-\tilde v\|_{L^2(\Omega_1)} \leq \epsilon \|\tilde v\|_{H^1(\tilde \Omega_1)}, \qquad \|u\|_{H^{\half}(\p B_2)}\leq C k^s \epsilon^{-\nu}\| \tilde v\|_{L^2(\Omega_1)}.
\end{align*}
\end{thm}
These results will be derived by duality from improved unique continuation estimates for the dual equations. The choice of the specific geometry here should be viewed as a sample results which -- based on the known unique continuation properties -- are expected for a larger class of convex domains.

\subsection{Connection with the literature}
\label{sec:literature}
In order to put our results into a proper context, we recall some of the earlier literature on improved stability properties. Due to their ability to stabilize notoriously ill-posed inverse problems, the stabilization effects at high frequency which had first been established in \cite{HI04} in the context of improved (interior) unique continuation properties were subsequently extended to improved unique continuation properties in various other geometric settings and other model equations \cite{IK11,I19,I07,IS10,IS07,CI16,I18,I11,EI18,EI20,IN14,IN14a,ILW26,I13,IW20,BLT10,
NUW13,BLZ20,INUW14}. The optimality of exponential $k$-dependences in unique continuation (in the form of three balls inequalities) was further  established recently in \cite{BM20} for the exact Helmholtz equation (which can be studied by investigating explicit behaviour of Bessel functions). Earlier, in \cite{J60}, the role of the geometry had already been highlighted for the closely connected wave equation, see also \cite{KRS20} for a systematic, microlocal argument for this.

Relying on these ideas further stability improvement results were also obtained for nonlinear inverse problems such as various variants of the Calder\'on problem. In this context, full data results were established in \cite{INUW14} for the Helmholtz equation with potential and in \cite{NUW13} for the acoustic equation. In recent work \cite{KU19}, this was extended to a partial data result for the Helmholtz equation with potential and impedance boundary conditions. Optimality of the improved stability estimates was discussed in a series of articles \cite{Isaev13,IN12,Isaev13a}.

\subsection{Outline of the remaining article} 
The remaining article is organised as follows: After briefly recalling some auxiliary results in Section \ref{sec:not}, we turn to the quantitative unique continuation results without geometric assumptions in Section \ref{sec:qUCP_without_geo}. In Section \ref{sec:Runge} a duality argument is used to transfer these into quantitative Runge approximation results. As an application we prove partial data stability for the Calder\'on problem for the acoustic equation with a priori information in a boundary layer in Section \ref{sec:stability}. Finally, in Section \ref{sec:improved_results_convexity} we discuss improvements arising from convex geometries.

\subsection{Notation and preliminaries}
\label{sec:not}
Before turning to the proofs of our main results we recall a number of auxiliary arguments and summarize our notation.

\subsubsection{On spectral estimates}

The following result contains a global estimate for the homogeneous Dirichlet problem depending on $\dist(k^2,\Sigma_{V,q})$.
It generalizes 
\cite[Proposition 2]{Beretta} and together with the assumption \textnormal{(\hyperref[assSpec]{a1})} allows us to invert the operator under consideration.

\begin{lem}\label{lem:apriori}
 Let $\Omega\subset\R^n$ be a bounded Lipschitz domain. Let $V$ and $q$ be as in \textnormal{(\hyperref[assV]{i})}-\textnormal{(\hyperref[assq]{ii})} in $\Omega$. Then there is a discrete set $\Sigma_{V,q}\subset\R$ such that for every $k^2\notin\Sigma_{V,q}$, $k\geq 1$, there exists a unique solution $u\in H^1(\Omega)$ of 
\begin{align}
\label{eq:beretta}
\begin{split}
\big(\Delta+k^2q+V\big)u &= f \;\mbox{ in } \Omega,\\
u &= 0 \; \mbox{ on }\p\Omega,
\end{split}
\end{align}
where $f\in L^2(\Omega)$.
In addition, there is a constant $C>0$ depending on $\Omega$, $\kappa$ and $\|V\|_{L^\infty(\Omega)}$ such that
\begin{align*}
\|u\|_{H^1(\Omega)}\leq C\left(1+\frac{k^3}{\dist(k^2,\Sigma_{V,q})}\right)\|f\|_{L^2(\Omega)}.
\end{align*}

\end{lem}

\begin{proof}
Recalling that zero is not a Dirichlet eigenvalue of $\D + V$ in $\Omega$, we consider the operator $T=(-\Delta-V)^{-1}M_q:H_0^1(\Omega)\to H_0^1(\Omega)$, where $M_q$ denotes  the multiplication operator $M_qu=qu$. 
Then $T$ has  eigenvalues $\alpha_n \in \R$ with $\alpha_n\to 0$ as $n\to\infty$.
Let $\Sigma_{V,q}=\{\lambda_n=\alpha_n^{-1}\}_{n\in\N}$ and let $\{e_n\}_{n\in\N}$ be an orthonormal basis of $L^2(\Omega)$ with $Te_n=\alpha_n e_n$.

Notice that \eqref{eq:beretta} is equivalent to $(I-k^2 T)u=(-\Delta-V)^{-1}f=:h$, $u\in H^1_0(\Omega)$. If $k^2\notin\Sigma_{V,q}$, by the Fredholm alternative, there is a unique solution $u$ to this problem.
Moreover, we can write  $u=\sum_{n\in\N} u_n e_n$, where  $u_n=( u,e_n )_{L^2(\Omega)}$ is given by
\begin{align*}
(1-k^2\alpha_n)u_n=h_n= ( h,e_n )_{L^2(\Omega)} \quad \mbox{ i.e.} \quad u_n= \frac{1}{1-{k^2}{\alpha_n}}h_n=\left(1+\frac{k^2}{\lambda_n-k^2}\right)h_n.
\end{align*}
Therefore,
\begin{align*}
\|u\|_{L^2(\Omega)}^2=\sum_{n\in\N}|u_n|^2
\leq \left(1+\frac{k^2}{\dist(k^2,\Sigma_{V,q})}\right)^2\sum_{n\in\N}|h_n|^2
=\left(1+\frac{k^2}{\dist(k^2,\Sigma_{V,q})}\right)^2\|h\|^2_{L^2(\Omega)}.
\end{align*}
Taking into account that $\|h\|_{L^2(\Omega)}\leq C\|f\|_{L^2(\Omega)}$, we conclude 
\begin{align*}
\|u\|_{L^2(\Omega)}^2
\leq C\left(1+\frac{k^2}{\dist(k^2,\Sigma_{V,q})}\right)\|f\|_{L^2(\Omega)}.
\end{align*}
Finally, testing the equation with $u$, we obtain $\|\nabla u\|_{L^2(\Omega)}\leq C(1+k)\|u\|_{L^2(\Omega)}+\|f\|_{L^2(\Omega)}$, which  in combination with the previous estimate yields the desired result.
\end{proof}

\subsubsection{Notation}

For $s\in \R$, the whole space Sobolev spaces are denoted by 
\begin{align*}
H^s(\R^n):=\big\{f\in \mathcal{S'}(\R^n): \|(1+|\cdot|^2)^{\frac s 2}\mathcal F f \|_{L^2(\R^n)} < \infty\big\},
\end{align*}
where
\begin{align*}
\mathcal F f(\xi)=\int_{\R^n} f(x)e^{-ix\cdot\xi}dx
\end{align*}
denotes the Fourier transform.

Let $\Omega\subset \R^n$ be an open set, then we define
\begin{align*}
H^s(\Omega)&:=\big\{f|_\Omega: f\in H^s(\R^n)\big\}, \mbox{ equipped with the quotient topology},\\
\tilde H^s(\Omega)&:=\mbox{ closure of } C^\infty_c(\Omega) \mbox{ in } H^s(\R^n).
\end{align*}
For any $s\in\R$ these spaces satisfy
\begin{align*}
\big(H^s(\Omega)\big)^*=\tilde H^{-s}(\Omega), \ \ \big(\tilde H^s(\Omega)\big)^*= H^{-s}(\Omega).
\end{align*}
In addition, for $\Gamma\subset\p\Omega$, we set
\begin{align*}
\tilde H^{\half}(\Gamma)&:=\big\{ f\in H^{\half}(\p\Omega): \supp f\subseteq\Gamma\big\},
\end{align*}
which is a closed subspace of $H^{\half}(\p\Omega)$ and its dual space may be identified with $ H^{-\half}(\Gamma)$.
We denote by $(\cdot,\cdot)_{L^2(\Omega)}$ the inner product in $L^2(\Omega)$ and also use the abbreviation $(\cdot, \cdot)_{\partial \Omega}$ to denote $(\cdot, \cdot)_{L^2(\partial \Omega)}$.

Furthermore, for $r>0$ and  $x_0\in\R^n$, we denote  the $n-$dimensional ball by $B_r(x_0)\subset\R^n$  and we define the cylindrical $(n+1)$-dimensional domain $Q_r(x_0):=B_r(x_0)\times(-r, r)\subset \R^{n+1}$. In addition, given an open set $\Omega\subset\R^n$, $B_r^+(x_0):=B_r(x_0)\cap\Omega$.

\section{Quantitative Unique Continuation}
\label{sec:qUCP_without_geo}

In this section we begin our analysis of the Runge approximation properties for the acoustic Helmholtz equation by proving quantitative unique continuation results without geometric assumptions on the underlying domains. Here we only assume the validity of the conditions \textnormal{(\hyperref[assV]{i})}-\textnormal{(\hyperref[assq]{ii})} (not necessarily the condition \textnormal{(\hyperref[assSpec]{a1})}). The Runge approximation properties will be deduced as dual results in the next section.
Since in this case exponential losses in $k$ are expected to be unavoidable (see \cite{BM20} for a proof of this in the closely related three balls inequalities), we do not prove these estimates by carefully tracking the $k$-dependence in the original equations but by embedding these equations into a family of elliptic equations without a large parameter but in an additional dimension. This is achieved by passing from $u(x)$ to $\tilde{u}(x,t) = e^{kt}u(x)$. We emphasize that this is a well-known procedure (see for instance \cite{LL12} and the references therein). The corresponding unique continuation properties follow from well-known results in the literature (e.g. \cite{ARRV09}). The main novelty of this first part of our article -- in which we do not pose geometric assumptions on our domains -- are the quantitative (in $k$) Runge approximation results and the application of these to the stability of the partial data inverse problem which are deduced in the next sections.

Formulated for the original function $u$ the unique continuation properties read as follows:

\begin{prop}
\label{prop:quantitative_UCP}
Let $\Omega$ be an open, bounded, connected Lipschitz domain and let $\Gamma\subset \p\Omega$ be a non-empty relatively open subset. Let $V$ and $q$ be as in \textnormal{(\hyperref[assV]{i})}-\textnormal{(\hyperref[assq]{ii})} in $\Omega$. Let $u\in H^1(\Omega)$ be a solution to 
\begin{align}\label{eq:eqinOm}
\begin{split}
\Delta u+k^2qu+Vu&=0   \;\mbox{ in } \Omega,
\end{split}
\end{align}
and let $M$, $\eta$ be such that
\begin{align*}
\|u\|_{H^1(\Omega)} &\leq M,\\
\|u\|_{H^{\half}(\Gamma)}+\|\p_{\nu} u \|_{H^{-\half}(\Gamma)}&\leq \eta.
\end{align*}
Then there exist  a parameter $\mu\in(0,1)$ and a constant $C>1$ depending on $n, \Omega, \Gamma, \|V\|_{L^\infty(\Omega)}, \kappa$ and $\|q\|_{C^1(\Omega)}$   such that 
\begin{align}\label{eq:QUCPboundary}
\|u\|_{L^2(\Omega)}\leq C k\left|\log\left(\frac{\eta}{M+\eta}\right)\right|^{-\mu}({M+\eta}).
\end{align}
In addition, if $G$ is a bounded Lipschitz domain with $G\Subset \Omega$, then there exist a parameter $\nu\in(0,1)$ and a constant $C>1$ (depending on $n, \Omega, G, \Gamma, \|V\|_{L^\infty(\Omega)}, \kappa$ and $\|q\|_{C^1(\Omega)}$) such that
\begin{align}\label{eq:QUCPinterior}
\|u\|_{L^2(G)}\leq  Ce^{C k}\left(\frac{\eta}{M+\eta}\right)^\nu (M+\eta).
\end{align}
\end{prop}

As an auxiliary ingredient, the proof of Proposition ~\ref{prop:quantitative_UCP} uses the following three-balls (boundary-bulk) inequalities derived from \cite{ARRV09}: 

\begin{lem}\label{lem:3balls}
Under the same assumptions as in Proposition ~\ref{prop:quantitative_UCP},  there exist a parameter ${\alpha \in (0,1)}$  and a constant $C>1$ depending on $\Omega, \|V\|_{L^{\infty}(\Omega)}, \kappa$ and $\|q\|_{C^1(\Omega)}$ such that
\begin{align}\label{eq:3balls}
\|u\|_{L^2(B_{r}(x_0))} \leq C e^{C k} \|u\|_{L^2(B_{2r}(x_0))}^{1-\alpha} \|u\|_{L^2(B_{r/2}(x_0))}^{\alpha},
\end{align}
where $x_0 \in \Omega$ and $r>0$ are such that $B_{4r}(x_0)\subset \Omega $.
\\
In addition, there exist a parameter $\alpha_0 \in (0,1)$ and a constant $C>1$ depending on $\Omega, \Gamma, \|V\|_{L^{\infty}(\Omega)}, \kappa$ and $\|q\|_{C^1(\Omega)}$ such that
\begin{align}\label{eq:3ballsboundary}
\|u\|_{L^2(B_{r}^+(x_0))} \leq C e^{C k} \big(\|u\|_{L^2(B_{2r}^+(x_0))}+\eta\big)^{1-\alpha_0} \eta^{\alpha_0},
\end{align}
where $x_0 \in \Gamma$ and $r>0$ are such that $B_{4r}(x_0)\cap (\partial \Omega \backslash\Gamma) = \emptyset$ 
and $B^+_r(x_0)=B_r(x_0)\cap\Omega$.
\end{lem}

In order to invoke the quantitative uniqueness results for elliptic equations without a large parameter, we pass to equations in an additional dimension which is a well-known method in quantitative uniqueness for eigenfunctions \cite{LL12, L18}. We remark that in the setting of Helmholtz equations where $q=1$, the $k$-dependence in this result is optimal as proved in \cite{BM20}.

\begin{proof}[Proof of Lemma ~\ref{lem:3balls}]
Let $\tilde \Omega :=\Omega\times(-d,d)$, with $d=\diam(\Omega)$.  Let $\tilde u\in H^1(\tilde\Omega)$ be a solution to 
\begin{align*}
\big(\D+q(x)\p_t^2+V(x)\big)\tilde u(x,t)=0 \quad \mbox{for}\quad (x,t)\in\tilde \Omega,
\end{align*}
where $\D$ denotes the Laplacian in $x$. Recalling the assumption \textnormal{(\hyperref[assq]{ii})}, we observe that the operator $\D + q(x)\p_t^2$ is elliptic with $C^1$ coefficients. Hence, the results from \cite{ARRV09} are applicable.
Using the notation $Q_r(x_0)=B_r(x_0)\times(-r,r)$,
by \cite[Theorem 1.10]{ARRV09}, there exist $C>1$ and $\alpha\in(0,1)$ depending on $\Omega, \|V\|_{L^\infty(\Omega)}, \kappa$ and $\|q\|_{C^1(\Omega)}$ such that
\begin{align}\label{eq:3cylinders}
\|\tilde u\|_{L^2(Q_r(x_0))}
\leq C \|\tilde u\|_{L^2(Q_{2r}(x_0))}^{1-\alpha}\|\tilde u\|_{L^2(Q_{\sfrac{r}{2}}(x_0))}^\alpha,
\end{align}
where $B_{4r}(x_0)\subset\Omega$.

We now consider the particular solution $\tilde u(x,t)=e^{kt}u(x)$, with $u(x)$ satisfying \eqref{eq:eqinOm}. 
Then \eqref{eq:3balls} follows from \eqref{eq:3cylinders}  
together with the observation that
\begin{align*}
2r e^{-kd} \|u\|_{L^2(B_r(x_0))}\leq \|\tilde u\|_{L^2(Q_r(x_0))}
\leq 2r e^{kd} \|u\|_{L^2(B_r(x_0))}.
\end{align*}
Inserting this into \eqref{eq:3cylinders} concludes the proof of \eqref{eq:3balls}.

In order to obtain \eqref{eq:3ballsboundary}, we observe that similarly, by \cite[Theorem 1.7]{ARRV09}, there exist $C>1$ and $\alpha_0\in(0,1)$ depending on $\Omega, \Gamma,  \|V\|_{L^\infty(\Omega)}, \kappa$ and $\|q\|_{C^1(\Omega)}$ such that
\begin{align*}
\|\tilde u\|_{L^2(Q_r(x_0)\cap \tilde\Omega)}
\leq C \big(\|\tilde u\|_{L^2(Q_{2r} \cap \tilde{\Omega})}+\tilde\eta\big)^{1-\alpha_0} \tilde\eta^{\alpha_0}.
\end{align*}
Here  $\tilde \Gamma=\Gamma\times(-d,d)$ and
\begin{align*}
\|\tilde u\|_{H^{\half}(\tilde\Gamma)}+\|\p_{\nu} \tilde  u \|_{H^{-\half}(\tilde\Gamma)}\leq\tilde\eta.
\end{align*}
The requirement  $\dist(Q_r(x_0)\cap \tilde\Omega, \p\tilde\Omega\backslash\tilde \Gamma)>0$  is satisfied since $B_{4r}(x_0)\cap (\partial \Omega \backslash\Gamma) = \emptyset$.

Notice that on the lateral boundary $\tilde{\Gamma}$ the normal derivative does not have any contribution in the $t$ direction. Therefore, using the definition of the weak form of $\p_{\nu}\tilde{u}$ in terms of the bilinear form associated with \eqref{eq:eqinOm}, choosing $\tilde u(x,t)=e^{kt}u(x)$ and  $\tilde\eta=Ce^{Ck}\eta$,  the estimate \eqref{eq:3ballsboundary} follows as above.
\end{proof}

With Lemma ~\ref{lem:3balls} available, we next address the proof of Proposition ~\ref{prop:quantitative_UCP}.

\begin{proof}[Proof of Proposition ~\ref{prop:quantitative_UCP}]
Let us define 
\begin{align}\label{eq:sets}
\begin{split}
W_\epsilon& : =\{x\in \Omega:  \dist(x, \p\Omega)<\epsilon\},\\
\Omega_\epsilon& : =\{x\in \Omega:  \dist(x, \p\Omega)\geq\epsilon\},
\end{split}
\end{align}
for $\epsilon\in (0, \epsilon_0)$, for some $\epsilon_0<1$ such that $\Omega_{\epsilon_0}$ is connected.
 We argue in three steps, estimating $u$ separately on $W_\epsilon $ and on $\Omega_\epsilon$ and combining these bounds by means of a final optimization step (in $\epsilon$).

\textit{Step 1: Estimate on $W_\epsilon$.} By  the H\"older and Sobolev inequalities we have
\begin{align}\label{eq:QUCstep1}
\|u\|_{L^2(W_\epsilon)}
\leq C\epsilon^{\frac{1}{p}}\|u\|_{L^q(W_\epsilon)}
\leq C \epsilon^{\frac{1}{p}}\|u\|_{H^1(\Omega)}
\leq C \epsilon^{\frac{1}{p}}M,
\end{align}
with $\frac{1}{p}+\frac{1}{q}=\frac{1}{2}$ and the constant $C>0$ depending on $\Omega$.

\textit{Step 2: Estimate on $\Omega_\epsilon$.}
We use Lemma ~\ref{lem:3balls} to propagate the smallness of $\eta$ to $\|u\|_{L^2(\Omega_\epsilon)}$.
Firstly, we transport the information from the boundary to the interior of $\Omega_\epsilon$. Let $x_0 \in \Gamma$ and $r_0>0$ such that $B_{4r_0}(x_0)\cap (\p\Omega\backslash\Gamma)=\emptyset.$
Then by  \eqref{eq:3ballsboundary}
 it holds
\begin{align*}
\|u\|_{L^2(B_{r_0}^+(x_0))}\leq C e^{Ck}
(M+\eta)^{1-\alpha_0}
\eta^{\alpha_0}.
\end{align*}

Once we have reached the interior, we iterate \eqref{eq:3balls} along a chain of balls which cover $\Omega_\epsilon$ and such that  $B_{4r}(x)\subset \Omega$. This implies that it is necessary to iterate \eqref{eq:3balls} roughly $N\sim N_0-C\log \epsilon$ times, where $N_0$ and $C$ depend on $\Omega$. 
Therefore, we obtain
\begin{align}\label{eq:QUCstep2}
\|u\|_{L^2(\Omega_\epsilon)}
\leq C e^{\frac{C}{1-\alpha}k}(M+\eta)^{1-\alpha_0\alpha^N} \eta^{\alpha_0\alpha^N}.
\end{align}

\textit{Step 3: Optimization.}
Combining  \eqref{eq:QUCstep1} and \eqref{eq:QUCstep2}, we obtain
\begin{align*}
\|u\|_{L^2(\Omega)}
\leq C\left(\epsilon^{\frac{1}{p}}+e^{Ck}\left(\frac{\eta}{M+\eta}\right)^{C_1\epsilon^{C_2}}\right)(M+\eta),
\end{align*}
where the constants depend on $\Omega, \Gamma, \|V\|_{L^\infty(\Omega)}, \kappa$ and $\|q\|_{C^1(\Omega)}$.
Abbreviating 
\begin{align}\label{eq:renaming}
\tilde{\epsilon}:= \epsilon^{C_2}, \quad \tilde{\eta}:= \left(\frac{\eta}{\eta + M} \right)^{C_1}, \quad \gamma:= \frac{1}{C_2 p},
\end{align}
we thus seek to optimize the expression
\begin{align*}
F(\tilde{\epsilon}, \tilde{\eta}):= \tilde{\epsilon}^{\gamma} + e^{Ck} \tilde{\eta}^{\tilde{\epsilon}}
\end{align*}
by choosing $\tilde{\epsilon} = \tilde{\epsilon}(\tilde{\eta})>0$ appropriately. Setting $\tilde{\epsilon}:= \frac{1}{(-\log\tilde{\eta})^{\beta}} + \frac{C k}{(-\log\tilde{\eta})}>0 $ for some $\beta \in (0,1)$, we obtain
\begin{align*}
|F(\tilde{\epsilon}, \tilde{\eta}) | \leq \left( \frac{1}{(-\log\tilde{\eta})^{\beta}} + \frac{k}{(-\log\tilde{\eta})} \right)^{\gamma} + e^{-|\log\tilde{\eta}|^{1-\beta}} \leq \frac{C k^{\gamma}}{|\log\tilde{\eta}|^{\beta \gamma}}+\frac{1}{|\log\tilde{\eta}|^{1-\beta}}.
\end{align*}
By \eqref{eq:renaming} we have $\gamma<1$ provided $p>2$ in \eqref{eq:QUCstep1} is chosen big enough. Then, for $k\geq 1$, in particular $k^\gamma<k$.
Choosing $\beta=\frac{1}{1+\gamma}$ we infer \eqref{eq:QUCPboundary} with $\mu=\beta\gamma<1$.
The bound \eqref{eq:QUCPinterior} follows directly from Step 2 for a  suitable choice of $\epsilon$ with $\nu=\alpha_0\alpha^{N{(\epsilon)}}$.
\end{proof}

\section{Proof of the Runge Approximation Theorems ~\ref{thm:Rungeboundary} and ~\ref{thm:Rungeinterior}}
\label{sec:Runge}

This section is devoted to the proofs of the (in $k$) quantitative Runge approximation results of Theorems ~\ref{thm:Rungeboundary} and ~\ref{thm:Rungeinterior}. This relies on duality arguments and the quantitative unique continuation results from the previous section.
In addition to the assumptions \textnormal{(\hyperref[assV]{i})} and \textnormal{(\hyperref[assq]{ii})}, we will now always also assume the condition \textnormal{(\hyperref[assSpec]{a1})} in $\Omega_2$ throughout the whole section in order to avoid solvability issues.

\begin{prop}\label{prop:QUCH1}
Let $\Omega_1\Subset \Omega_2$ and $\Gamma\subset\p\Omega_2$ be as in Theorem ~\ref{thm:Rungeboundary}.
Let $V$ and $q$ satisfy the assumptions \textnormal{(\hyperref[assV]{i})}-\textnormal{(\hyperref[assq]{ii})} in $\Omega_2$.
Let $u\in H^1(\Omega_2)$ be the unique solution to 
\begin{align*}
\begin{split}
\Delta u+k^2qu+Vu&=v\mathbb{1}_{\Omega_1} \;  \mbox{ in } \Omega_2,\\
u&=0 \!\quad\quad\mbox{ on } \p\Omega_2,
\end{split}
\end{align*}
with $v\in L^2(\Omega_1)$ and $k\geq 1$ satisfying the condition \textnormal{(\hyperref[assSpec]{a1})}.
Then there exist  a parameter $\mu_0\in(0,1)$ and a constant $C>1$ depending on $n, \Omega_2, \Omega_1, \Gamma, \|V\|_{L^\infty(\Omega_2)}, \kappa$ and $\|q\|_{C^1(\Omega_2)}$ such that 
\begin{align}\label{eq:QUCPgradboundary}
\|u\|_{H^1(\Omega_2\backslash\overline\Omega_1)}\leq C k^{n+4}\left|\log\left(C\frac{\|\p_\nu u\|_{H^{-\half}(\Gamma)}}{\|v\|_{L^2(\Omega_1)}}\right)\right|^{-\mu_0}\|v\|_{L^2(\Omega_1)}.
\end{align}
In addition, if $G$ is a bounded Lipschitz domain with $G\Subset \Omega_2\backslash\Omega_1$, then there exist a parameter $\nu_0\in(0,1)$ and a constant $C>1$ depending on $n, \Omega_2, \Omega_1, G, \Gamma, \|V\|_{L^\infty(\Omega_2)}, \kappa$ and $\|q\|_{C^1(\Omega_2)}$ such that
\begin{align}\label{eq:QUCPgradinterior}
\|u\|_{H^1(G)}\leq  Ce^{C k}\left(\frac{\|\p_\nu u\|_{H^{-\half}(\Gamma)}}{\|v\|_{L^2(\Omega_1)}}\right)^{\nu_0}
\|v\|_{L^2(\Omega_1)}.
\end{align}
\end{prop}

\begin{proof}

We start by estimating $\|u\|_{H^1(\Omega_2)}$ in terms of $v$. 
By Lemma ~\ref{lem:apriori}, there is a constant $C>1$ such that
\begin{align*}
\|u\|_{H^1(\Omega_2)}\leq C\left(1+\frac{k^{3}}{\dist(k^2, \Sigma_{V,q})}\right)\|v\|_{L^2(\Omega_1)}
\leq Ck^{n+1} \|v\|_{L^2(\Omega_1)},
\end{align*}
where for the last inequality we have used the assumption \textnormal{(\hyperref[assSpec]{a1})}.

Since $u$ satisfies \eqref{eq:generaleq} in $\Omega=\Omega_2\backslash\overline\Omega_1$, which is connected, the results of Proposition ~\ref{prop:quantitative_UCP} hold with 
\begin{align}\label{eq:selectionMeta}
\begin{split}
\eta&=\|\p_\nu u\|_{H^{-\half}(\Gamma)}, \quad M=Ck^{n+1}\|v\|_{L^2(\Omega_1)}.
\end{split}
\end{align}

In order to promote \eqref{eq:QUCPboundary} and \eqref{eq:QUCPinterior} to the gradient, we argue similarly as in Proposition ~\ref{prop:quantitative_UCP}. 
We consider the subsets $W_\epsilon$ and $\Omega_\epsilon$ defined in \eqref{eq:sets} with  
$\Omega=\Omega_2\backslash\overline\Omega_1$.

\textit{Step 1': Estimate on $W_\epsilon$.}
By the H\"older inequality
\begin{align*}
\|\nabla u\|_{L^2(W_\epsilon)}\leq C \epsilon^{\frac{1}{p}}\|\nabla u\|_{L^q(\Omega_2)},
\end{align*}
where $\frac 1 p+\frac 1 q=\frac 1 2$. By \cite[Theorem 1]{M63} (together with \cite[Theorem 0.5]{JK95} for the admissibility of the Lipschitz domain) there exists $q>2$ such that 
\begin{align*}
\|\nabla u\|_{L^q(\Omega_2)}\leq C\|v\mathbb{1}_{\Omega_1}-k^2 qu-Vu\|_{L^2(\Omega_2)}
\leq C\Big(\|v\|_{L^2(\Omega_1)}+(k^2 \kappa+\|V\|_{L^\infty(\Omega_2)})\|u\|_{L^2(\Omega_2)}\Big).
\end{align*}
In addition, testing the weak version of the equation with itself, we have  
\begin{align}\label{eq:kL2}
k\|u\|_{L^2(\Omega_2)}\leq C(\|u\|_{H^1(\Omega_2)}+\|v\|_{L^2(\Omega_1)})\leq CM,
\end{align}
with $C$ depending on $\kappa$ and $\|V\|_{L^\infty(\Omega_2)}$.
Therefore,
\begin{align*}
\|\nabla u\|_{L^2(W_\epsilon)}\leq C \epsilon^{\frac{1}{p}}k M.
\end{align*}

\textit{Step 2': Estimate on $\Omega_\epsilon $.}
Let $\chi$ be a smooth cut-off function supported in $\Omega_{\epsilon/2}$ with $\chi=1$ in $\Omega_\epsilon$ and $|\nabla \chi|\leq c\epsilon^{-1}$. We then obtain the following Caccioppoli inequality by testing the equation with $\chi^2 u$:
\begin{align*}
\|\nabla u\|_{L^2(\Omega_\epsilon)}
&\leq C(\epsilon^{-1}+\|V\|_{L^\infty(\Omega_2)}^{\frac 1 2}+k\kappa)\|u\|_{L^2(\Omega_{\epsilon/2})}
\\&\leq C k\epsilon^{-1}\|u\|_{L^2(\Omega_{\epsilon/2})}.
\end{align*}
Inserting the estimate \eqref{eq:QUCPinterior} with explicit $\epsilon$ dependence coming from the Step 2 in the proof of Proposition ~\ref{prop:quantitative_UCP}, we infer
\begin{align*}
\|\nabla u\|_{L^2(\Omega_\epsilon)}
\leq Ck\epsilon^{-1} e^{Ck}
\left(\frac{\eta}{M+\eta}\right)^{C_1'\epsilon^{C_2}}(M+\eta).
\end{align*}

\textit{Step 3': Optimization.}
Combining the previous two steps we obtain
\begin{align*}
\|\nabla u\|_{L^2(\Omega_2\backslash\overline\Omega_1)}\leq C  k \left(\epsilon^{\frac{1}{p}}+\epsilon^{-1} e^{Ck}\left(\frac{\eta}{M+\eta}\right)^{C_1'\epsilon^{C_2}}\right)(M+\eta).
\end{align*}
Optimizing in $\epsilon$ as in the Step 3 in the  proof of Proposition ~\ref{prop:quantitative_UCP} yields
\begin{align*}
\|\nabla u\|_{L^2(\Omega_2\backslash\overline\Omega_1)}
\leq C k^2\left|\log\left(\frac{\eta}{M+\eta}\right)\right|^{-\mu_0}(M+\eta)
\end{align*}
for a suitable $\mu_0\in(0,1)$.

Introducing \eqref{eq:selectionMeta} and taking into account that by \eqref{eq:kL2}
\begin{align}\label{eq:estetaM}
\begin{split}
\eta=\|\p_\nu u\|_{H^{-\half}(\Gamma)}
&\leq C(k^2 \|u\|_{L^2(\Omega_2)}+\|\nabla u\|_{L^2(\Omega_2)}+\|v\|_{L^2(\Omega_1)})
\\
&\leq CkM
\leq Ck^{n+2}\|v\|_{L^2(\Omega_1)},
\end{split}
\end{align}
we infer \eqref{eq:QUCPgradboundary}.

Estimate \eqref{eq:QUCPgradinterior} follows from  the Caccioppoli inequality in Step 2' for suitable choice of $\epsilon$ together with the previous estimate for $\eta$.
\end{proof}

Using the results from Proposition ~\ref{prop:QUCH1} we now address the proof of Theorem ~\ref{thm:Rungeboundary}:

\begin{proof}[Proof of Theorem ~\ref{thm:Rungeboundary}]
We seek to show that for any $\alpha >0$ there exists a solution $u_\alpha$ to
\begin{align*}
(\Delta+k^2q+V)u_\alpha = 0 \; \mbox{ in } \Omega_2
\end{align*} with
\begin{align*}
\|u_\alpha-v\|_{L^2(\Omega_1)} \leq C(\alpha, k)\|v\|_{H^1(\Omega_1)}, \qquad \|u_\alpha\|_{H^{\half}(\p \Omega_2)} \leq \frac{1}{\alpha}\| v\|_{L^2(\Omega_1)}.
\end{align*}

Let $X$ be the closure of $\{u\in H^1(\Omega_1)\mid(\Delta  + k^2q+V)u=0 \text{ in } \Omega_1\}$ in $L^2(\Omega_1)$. 
We then define 
\begin{align*}
A \colon \tilde  H^{\half}(\Gamma)&\to X,
\\
g \quad&\mapsto Ag = u|_{\Omega_1},
\end{align*} 
where $u\in  H^1(\Omega_2)$ is the solution to \eqref{eq:generaleq} in $\Omega_2$ satisfying the boundary condition $u|_{\p\Omega_2}= g\in \tilde  H^{\half}(\Gamma)$. 
We denote by $A^*$ the Hilbert space adjoint of $A$, which maps 
\begin{align*}
A^* \colon X &\to \tilde  H^{\half}(\Gamma),
\\
u &\mapsto A^*u = R(\p_\nu w|_\Gamma),
\end{align*} 
where $R$ is the Riesz isomorphism $R:  H^{-\half}(\Gamma)\to \tilde H^{\half}(\Gamma)$ and $w\in H^1(\Omega_2)$ satisfies 
\begin{align*}
(\Delta+k^2q+V)w&=\mathbb{1}_{\Omega_1}u \;\mbox{ in } \Omega_2,\\
w&=0 \qquad\mbox{ on }  \p \Omega_2.
\end{align*}
By \cite[Lemma 4.1]{RS17}, $A$ is a compact, injective operator with dense range in $X$ and applying the spectral theorem to $A^*A$ yields an orthonormal basis of eigenvectors $\{\phi_j\}_{j=1}^\infty$ for $ \tilde H^{\half}(\Gamma)$ and a sequence of positive, decreasing eigenvalues $\{\mu_j\}_{j=1}^\infty$ with 
\[A^*A\phi_j=\mu_j\phi_j.\] Then, setting $\psi_j:=\mu_j^{-\half} A\phi_j$ yields an orthonormal basis $\{\psi_j\}_{j=1}^\infty$ of $X$. In particular, we have \begin{align}\label{eq:adjAbasis}
    \|A^*\psi_j\|_{ H^{\half}(\partial \Omega_2)}\leq \mu_j^{\half}.
\end{align}

Returning to our setting, we notice that $v\in X$, hence it admits a unique decomposition in the orthonormal basis $v=\sum_{j=1}^\infty \beta_j\psi_j$. For $\alpha>0$, we define 
$$v_\alpha:=\sum_{\alpha \geq \mu_j^{\half}}\beta_j\psi_j$$ and let
$w_\alpha$ be the solution of 
\begin{align*}
(\Delta +k^2q+V)w_\alpha&=\mathbb{1}_{\Omega_1}v_\al \;\mbox{ in } \Omega_2,\\
w_\alpha&=0 \;\qquad \mbox{ on } \partial \Omega_2.
\end{align*}
Here the notation $\alpha \geq \mu^{\half}_{j}$ is an abbreviation for the set $\{j\in \N: \  \alpha \geq \mu^{\half}_{j}\}$.
By \eqref{eq:adjAbasis}, it holds
\begin{align}
\label{eq:partial boundary bd2}
\|\partial_\nu w_\alpha\|_{H^{-\half}(\Gamma)}=\|A^*v_\alpha\|_{H^{\half}(\partial\Omega_2)}\leq \alpha \|v_\alpha\|_{L^2(\Omega_1)}.
\end{align}
Now, we define $u_\al$ as the solution to \eqref{eq:generaleq} on $\Omega_2$ satisfying the boundary condition $ u_\al = g_\al$ on $\p\Omega_2$, with $g_\al=\sum_{\alpha\le\mu_j^{\half}}\beta_j\mu_j^{-\half}\phi_j$. 
Note that at the boundary we have by the previous considerations 
\begin{align*}
\|u_\alpha\|_{H^{\half}(\p\Omega_2)}^2=\|g_\alpha\|_{ H^{\half}(\p\Omega_2)}^2=\bigg\|\sum_{\alpha\le\mu_j^{\half}}\beta_j\mu_j^{-\half}\phi_j\bigg\|^2_{ H^{\half}(\partial \Omega_2)}=\sum_{\alpha\le\mu_j^{\half}}\frac{\beta_j^2}{\mu_j}
    \leq \frac{1}{\alpha^2} \| v \|_{L^2(\Omega_1)}^2.
\end{align*}
In addition, notice that  
\begin{align*}
u_\alpha|_{\Omega_1}=Ag_\al=A\left(\sum_{\alpha\le\mu_j^{\half}}\beta_j\mu_j^{-\half}\phi_j\right)=\sum_{\alpha\le\mu_j^{1/2}}\beta_j\psi_j=v-v_\alpha.
\end{align*}
Thus, it remains to obtain an explicit dependence on $\alpha$ and $k$ in 
\begin{align*}
\|u_\alpha-v\|_{L^2(\Omega_1)}=\|v_\alpha\|_{L^2(\Omega_1)}\leq C(\alpha, k) \|v\|_{H^1(\Omega_1)}.
\end{align*}
Orthogonality considerations show 
\begin{align}\label{eq:valpha1}
\|v_\alpha\|^2_{L^2(\Omega_1)}=(v,\partial_\nu w_\alpha)_{\partial\Omega_1}-(\partial_\nu v,w_\alpha)_{\partial \Omega_1}.
\end{align}
Using trace estimates for the solutions we find
\begin{align}
\label{eq:auxw}
\|v_\alpha\|^2_{L^2(\Omega_1)}
&\leq C (1+\|V\|_{L^\infty(\Omega_2)}+k^2\kappa) \|v\|_{H^1(\Omega_1)}\|w_\alpha\|_{H^1(\Omega_2\backslash\overline\Omega_1)}
\end{align}
with some constant $C>0$ depending on $\Omega_1, \Omega_2\backslash \Omega_1$.  Using \eqref{eq:QUCPgradboundary} to estimate the norm of $w_{\alpha}$ in \eqref{eq:auxw} yields
\begin{align*}
\|v_\alpha\|_{L^2(\Omega_1)}^2
&\leq C k^{n+6}\left|\log\left(C\frac{\|\p_\nu w_{\alpha}\|_{H^{-\half}(\Gamma)}}{\|v_{\alpha}\|_{L^2(\Omega_1)}}\right)\right|^{-\mu_0}\|v_{\alpha}\|_{L^2(\Omega_1)}\|v\|_{H^1(\Omega_1)}.
\end{align*}
Finally, dividing by $\|v_{\alpha}\|_{L^2(\Omega_1)}$, recalling \eqref{eq:partial boundary bd2} and using monotonicity, we arrive at
\begin{align*}
\|v_\alpha\|_{L^2(\Omega_1)}
& \leq C k^{n+6}\left|\log (C\alpha) \right|^{-\mu_0}\|v\|_{H^1(\Omega_1)}.
\end{align*}
We choose $\alpha<1$ so that $C k^{n+6}{\left|\log (C\alpha) \right|^{-\mu_0}}=\epsilon$, i.e.
\begin{align*}
\frac{1}{\alpha}= C e^{ Ck^{s} \epsilon^{-\mu}}
\end{align*}
with $s=\frac{n+6}{\mu_0}$ and $\mu=\mu_0^{-1}$.
This concludes the proof.
\end{proof}

Relying on similar ideas, we also obtain the bounds from Theorem ~\ref{thm:Rungeinterior}:

\begin{proof}[Proof of Theorem ~\ref{thm:Rungeinterior}]
We define $v_\alpha$ and $w_\alpha$ as in  the proof of Theorem ~\ref{thm:Rungeboundary} with $v=\tilde v|_{\Omega_1}$.
Equations ~\eqref{eq:valpha1} and \eqref{eq:auxw} can be slightly modified to read
\begin{align*}
\|v_\alpha\|^2_{L^2(\Omega_1)}=(\tilde v,\partial_\nu w_\alpha)_{\partial\tilde \Omega_1}-(\partial_\nu \tilde v,w_\alpha)_{\partial\tilde \Omega_1}\leq Ck^2\|\tilde v\|_{H^1(\tilde\Omega_1)}\|w_\alpha\|_{H^1(G)},
\end{align*}
where $G=\Omega_2'\backslash \tilde\Omega_1\Subset\Omega_2\backslash\Omega_1$ with $\Omega_2'\Subset\Omega_2$. 
Arguing as above and using the  quantitative unique continuation result \eqref{eq:QUCPgradinterior}, we obtain
\begin{align*}
\|u_\alpha-u\|_{L^2(\Omega_1)}=\|v_\alpha\|_{L^2(\Omega_1)}
&\leq C e^{C k}\left(\frac{\|\partial_\nu w_\alpha\|_{H^{-\half}(\partial \Omega_2)}}{\|v_\alpha\|_{L^2(\Omega_1)}}\right)^{\nu_0} \|\tilde v\|_{H^1(\tilde \Omega_1)}
\\ &\leq C e^{C k}\alpha^{\nu_0}
\|\tilde v\|_{H^1(\tilde \Omega_1)}.
\end{align*}
Choosing $\alpha$ so that $ e^{C k}\alpha^{\nu_0}=\epsilon$, i.e. $\frac{1}{\alpha}= C\Big(\frac{e^{C k}}{\epsilon}\Big)^{\frac1{\nu_0}}$,
the result follows with $\nu={\nu_0}^{-1}$. 
\end{proof}

\subsection{Optimality of the estimates in Theorem ~\ref{thm:Rungeinterior}}
\label{sec:optimality}
In order to infer the optimality of the quantitative Runge approximation results in the parameter $\epsilon$, we consider the case $q=1$ (i.e. the case of the Helmholtz equation). 
We remark that optimality results in $k$ for three balls inequalities were recently obtained in \cite{BM20}.

\begin{lem}
Let $\Omega_2=B_1, \Omega_1=B_{\half}$ and $\Gamma=\p B_1$.
For fixed $k\geq 1$, there exists $N=N(k)\in\N$ and a sequence $(v_\ell)_{\ell\geq N}$ of solutions to $(\Delta+k^2)v_\ell=0$ in $\Omega_1$ with $\|v_\ell\|_{H^1(\Omega_1)}=1$ such that for any solution $u$ of $(\Delta+k^2)u=0$ in $\Omega_2$ with $\|v_\ell-u\|_{L^2(\Omega_1)}\leq(2^{\frac n2+4}\ell)^{-1}$ we have $\|u\|_{H^{\half}(\Gamma)}\geq ce^{C\ell}$.
\end{lem}

\begin{proof}
Arguing by separation of variables, we obtain that any solution $u\in H^1(B_1)$ of ${(\Delta+k^2)u=0}$ can be written with respect to the variables $r=|x|$, $\theta= \frac{x}{|x|}\in \mathbb{S}^{n-1}$ as
\begin{align*}
u(x)=u(r,\theta)=\sum_{\ell=0}^\infty \sum_{m=0}^{N_\ell}  c_{\ell m} R_\ell(k r) \psi_{\ell m}(\theta),
\end{align*}
where $\{\psi_{\ell m}\}_{m=0}^{N_\ell}$ is an orthonormal basis of $L^2(\mathbb{S}^{n-1})$ consisting of the spherical harmonics of degree $\ell$ and 
$$R_\ell(r)= r^{1-\frac n 2} J_{\ell+\frac n 2 -1}(r),$$ with $J_\alpha$ denoting the Bessel functions.

We consider 
$g_\ell(x)=R_\ell(k r) \psi_{\ell,1}(\theta)$
and define 
$v_\ell=\alpha_\ell g_\ell$ with $\alpha_\ell=\|g_\ell\|_{H^1(\Omega_1)}^{-1}$.
Then, we may write $u=cv_\ell+w$, where $(w, g_\ell)_{L^2(B_1)}=0$ and $c\alpha_\ell=\|g_\ell\|_{L^2(B_1)}^{-2}(u, g_\ell)_{L^2(B_1)}$.
Therefore, 
$$u(x)|_{\p B_1}=c\alpha_\ell R_\ell(k)\psi_{\ell,1}(\theta)+\omega(\theta)$$ with $(\omega, \psi_{\ell,1})_{L^2(\p B_1)}=0$ and $\omega(\theta) = w(x)|_{\partial B_1}$.

We are interested in estimating $\|u\|_{H^{\half}(\Gamma)}$ from below.
If we assume $\|u-v_\ell\|_{L^2(\Omega_1)}<\epsilon$, then
$|c-1|\alpha_\ell\|g_\ell\|_{L^2(\Omega_1)}<\epsilon$.
Therefore,
\begin{align}\label{eq:optimalitybdnorm}
\begin{split}
\|u\|_{H^\half(\Gamma)}
&\geq (1+\lambda_\ell^{\frac 1 2})^{\frac 1 2} |c\alpha_\ell||R_\ell(k)|
\geq \ell^{\frac{1}{2}}\big(|\alpha_\ell| -\epsilon\|g_\ell\|_{L^2(\Omega_1)}^{-1}\big) |R_\ell(k)|
\\ &\geq \ell^{\frac{1}{2}}\big(\|g_\ell\|_{H^1(\Omega_1)}^{-1}-\epsilon\|g_\ell\|_{L^2(\Omega_1)}^{-1}\big) |R_\ell(k)|,
\end{split}
\end{align}
where $\lambda_\ell=\ell(\ell+n-2)$.

Using \cite[10.22.27]{NIST}, we can estimate the  $L^2$ norm of $g_{\ell}$ as follows:
\begin{align*}
\|g_\ell\|_{L^2(\Omega_1)}^2
&=\int_0^{\frac 1 2} R_\ell^2(kr)r^{n-1} dr
=k^{-n}\int_0^{\frac k 2} t J_{\ell+\frac n 2-1}^2(t)dt
\\ &=2k^{-n}\sum_{m=0}^\infty \Big(\ell+\frac n 2+2m\Big) J_{\ell+\frac n 2+2m}^2\Big(\frac k 2 \Big)
\geq k^{-n}(2\ell+n) J_{\ell+\frac n 2}^2\Big(\frac k 2 \Big).
\end{align*}
For the norm of the gradient, using  integration by parts and the equation, we obtain 
\begin{align*}
\|\nabla g_\ell\|_{L^2(\Omega_1)}^2
&=\int_{\p B_{\half}} g_\ell \partial_r g_\ell d\mathcal H^{n-1}(\theta)-\int_{B_{\half}}g_\ell \Delta g_\ell  dx
\\&= R_\ell\Big(\frac{k}{2}\Big)\partial_r R_\ell\Big(\frac{k}{2}\Big)+k^2\|g_\ell\|^2_{L^2(\Omega_1)}.
\end{align*}

We next collect some properties of Bessel functions from \cite{Paris84}  for $x\in (0,1)$ and $\alpha>0$:
\begin{align}\label{eq:Besselprop}
1\leq \frac{J_{\alpha}(\alpha x)}{x^\alpha J_{\alpha}(\alpha)}\leq  e^{\alpha(1-x)}, 
\qquad
0<\frac 1 x-\frac{J'_{\alpha}(\alpha x)}{J_{\alpha}(\alpha x)} <1,
\qquad \frac{J_\alpha(\alpha x)}{J_{\alpha+1}(\alpha x)}<\frac{2\alpha+2}{\alpha x}.
\end{align}
By the second estimate  in \eqref{eq:Besselprop} and the fact that $J_\alpha(\alpha x)>0$ (e.g. \cite[10.14.2]{NIST} together with  the first estimate in \eqref{eq:Besselprop} or \cite[10.14.7]{NIST}), we have that $J_\alpha(\alpha x)$ is monotonously increasing for $x\in(0,1)$.
Moreover, we know  \cite[10.19.1]{NIST} that for $z\neq 0$ fixed
\begin{align}
\label{eq:Besselasym}
J_\alpha(z)\sim \frac{1}{\sqrt{2\pi\alpha}}\left(\frac{ez}{2\alpha}\right)^\alpha, \qquad \alpha\to\infty.
\end{align}

We assume from now on that $\ell+\frac n 2-1>k$, so the previous estimates  \eqref{eq:Besselprop} can be applied.
In particular, due to the monotonicity
\begin{align*}
\|g_\ell\|_{L^2(\Omega_1)}
=\left(\int_0^{\frac 1 2} R_\ell^2(kr)r^{n-1} dr\right)^{\frac 1 2}
\leq \frac{1}{\sqrt{2}} R_\ell\Big(\frac{k}{2}\Big),
\end{align*}
Inserting the previous estimates  on $g_\ell$ into \eqref{eq:optimalitybdnorm}
yields
\begin{align*}
\|u\|_{H^{\half}(\Gamma)}
&\geq \ell^{\frac 1 2}\left(\frac{\frac{R_\ell(k)}{R_\ell(\frac{k}{2})}}{1+k+ \Big(\frac{\p_rR_\ell(\frac{k}{2})}{R_\ell(\frac{k}{2})}\Big)^{\frac 1 2}}-\epsilon\frac{R_\ell(k)k^{\frac n 2}}{(2\ell+n)^{\frac 12}J_{\ell+\frac n2}\big(\frac k2\big)}\right)
\\
&= \ell^{\frac 1 2}\frac{J_\alpha(k)}{J_\alpha\big(\frac{k}2\big)}
\left(  \frac{2^{1-\frac n2}}{1+k+ \big(\frac{\p_rR_\ell(\frac{k}{2})}{R_\ell(\frac{k}{2})}\big)^{\frac 1 2}}-\epsilon\frac{k}{(2\ell+n)^{\frac 1 2}}\frac{J_\alpha\big(\frac k2\big)}{J_{\alpha+1}\big(\frac k2\big)}\right),
\end{align*}
where $\alpha=\ell+\frac n 2-1>k$.
Using the different estimates in \eqref{eq:Besselprop},  we deduce
\begin{align*}
\frac{\p_rR_\ell\big(\frac{k}{2}\big)}{R_\ell\big(\frac{k}{2}\big)}
&=\frac{J'_{\alpha}\big(\frac k 2\big)}{J_\alpha\big(\frac k 2\big)}-\frac{n-2}{k}<\frac{2\alpha}{k}-\frac{n-2}{k}=\frac{2\ell}{k},
\\
\frac{k}{(2\ell+n)^{\frac 1 2}}\frac{J_\alpha\big(\frac k2\big)}{J_{\alpha+1}\big(\frac k2\big)}
&\leq\frac{k}{(2\ell+n)^{\frac 1 2}}\frac{2\alpha+2}{\frac k 2}
= 2(2\ell+n)^{\frac 1 2}.\end{align*}
Therefore, 
\begin{align}\label{eq:optimalitybdnorm2}
\|u\|_{H^{\half}(\Gamma)}
&\geq 2\ell^{\frac 1 2}\frac{J_\alpha(k)}{J_\alpha\big(\frac{k}2\big)}
\left(\frac{2^{-\frac n2}}{1+k+\big(\frac{2\ell}{k}\big)^{\frac 1 2}}-\epsilon (2\ell+n)^{\frac 12}\right).
\end{align}

In order to finally obtain the  optimality in $\epsilon$, we consider $\ell\gg\max\{k^2, n\}$ and $\epsilon=(2^{\frac n 2+4}\ell)^{-1}$.
Then, by \eqref{eq:Besselasym} 
\begin{align*}
\|u\|_{H^{\half}(\Gamma)}
&\geq C2^{-\frac n2} 2^\alpha \ell^{\frac 1 2}
\left(\frac{1}{3\ell^{\frac 1 2}}-\frac{3}{16\ell^{\frac 1 2}}\right)
>c2^\ell. \qedhere
\end{align*}
\end{proof}

\section{Stability for the Calder\'on Problem for the Helmholtz Equation with Potential}
\label{sec:stability}

As an application of the Runge approximation results from above, we present the proof of the stability estimate from Proposition ~\ref{prop:stability} for a partial data Calder\'on problem with stability improvement for an increasing parameter $k$. For the Helmholtz setting with impedance boundary conditions this had earlier been deduced in \cite{KU19}.

More precisely, we assume the following set-up: We consider $n\ge 3$, $\Omega\subseteq \R^n$ a bounded connected open set with $C^\infty$ boundary and $V$ and $q$ as in \textnormal{(\hyperref[assV]{i})}-\textnormal{(\hyperref[assq]{ii})} and $k\geq 1$ satisfying \textnormal{(\hyperref[assSpec]{a1})}.
Let $\Gamma$ be a non-empty open subset of $\p\Omega$.
We study the  local  Dirichlet-to-Neumann map
\begin{align*}
\Lambda_{V,q}^\Gamma(k): \tilde H^{\half}(\Gamma) &\to H^{-\half}(\Gamma),\\
g & \mapsto \p_\nu u|_{\Gamma}, 
\end{align*}
where $u\in H^1(\Omega)$ is the solution to 
\begin{align*}
(\Delta+k^2 q+V)u&=0 \;\mbox{ in }\Omega,\\
u&=g \;\mbox{ on }\p\Omega.
\end{align*}

Theorem ~\ref{thm:Rungeinterior} allows us to obtain stability results for the inverse problem by using the strategy from \cite[Proposition 6.1]{RS17}. In particular, this reproves the result of \cite[Theorem 1.2]{KU19} in the case of Dirichlet boundary conditions and our spectral assumption \hyperref[assq]{(a1)}.

\begin{proof}[Proof of Proposition ~\ref{prop:stability}]
We will use the short hand notation $b_j=k^2q_j+V_j$ for $j=1,2$, for which $\|b_j\|_{L^2(\Omega)}\leq k^2B$  and $b_1=b_2$ in $\Omega\backslash\Omega'$.

We start with the construction of complex geometrical optics solutions $u_j$ solving $(\Delta+b_j)u_j = 0$ in $\Omega$ following  \cite{SU87}.
We fix $\omega\in \mathbb S^{n-1}$ and choose 
$\omega^\perp, \tilde\omega^\perp\in \mathbb S^{n-1}$ such that
\[\omega \cdot \omega^{\bot}
= \omega \cdot \widetilde{\omega}^{\bot}
= \omega^{\bot} \cdot \widetilde{\omega}^{\bot} = 0.\]
We set for $\tau, r\in \R$ with $\tau\geq \frac{|r|}{2}$
\begin{align*}
\xi_1=\tau\omega^\perp+i\left(-\frac{r}{2}\omega+\sqrt{\tau^2-\frac{r^2}{4}}\tilde\omega^\perp\right),
\quad
\xi_2&=-\tau\omega^\perp+i\left(-\frac{r}{2}\omega-\sqrt{\tau^2-\frac{r^2}{4}}\tilde\omega^\perp\right).
\end{align*}
By \cite{SU87}, if $\tau\geq \max\{C_0k^2B,1\}$, there are solutions $u_j$ for $j\in\{1,2\}$ of the form
\begin{align*}
u_j(x)=e^{\xi_j\cdot x}\big(1+\psi_j(x)\big),
\end{align*}
where
\begin{align*}
\|\psi_j\|_{L^2(\Omega)}\leq\frac{Ck^2B}{\tau}, 
\qquad
\|\psi_j\|_{H^1(\Omega)}\leq Ck^2B.
\end{align*}
This implies the following estimates for the solutions:
\begin{align*}
\|u_j\|_{L^2(\Omega)}\leq  Ce^\tau, \qquad 
\|u_j\|_{H^1(\Omega)}\leq  C\tau e^{\tau}.
\end{align*}

Now we seek to approximate $u_j$ up to some order $\epsilon>0$ which will be chosen later.
We apply  Theorem ~\ref{thm:Rungeinterior}
 with $\Omega_2=\Omega$, $\Omega_1=\Omega'$ and $\tilde\Omega_1=\tilde \Omega'$, where the latter is a slightly bigger domain containing $\Omega'$. This yields solutions $\tilde u_j$  to $ (\Delta+b_j)\tilde{u}_j = 0 $ in $\Omega$ with  $\tilde{u}_j|_{\p\Omega}$ supported in $\Gamma$ and
\begin{align*}
\|u_j-\tilde u_j\|_{L^2(\Omega')}\leq \epsilon\|u_j\|_{H^1(\tilde\Omega')}, \qquad \|\tilde u_j\|_{H^{\half}(\p\Omega)}\leq Ce^{Ck}\epsilon^{-\nu}\|u_j\|_{L^2(\Omega')}.
\end{align*}
In addition, since $b_1=b_2$ in $\Omega\backslash\Omega'$ and using integration by parts, we obtain the following analog to Alessandrini's identity \cite{A88}
\begin{equation*}
\int_{\Omega'} (b_2-b_1) \tilde u_{1} \tilde u_{2} \, d x =\int_{\Omega} (b_2-b_1) \tilde u_{1} \tilde u_{2} \, d x =\Big(\big(\Lambda_{V_1,q_1}^\Gamma(k)-\Lambda_{V_2,q_2}^\Gamma(k)\big)\tilde u_1,\tilde u_2\Big)_{L^2(\partial \Omega)}.
 \end{equation*}
Abbreviating $\delta:= \|\Lambda_{V_1,q_1}^\Gamma(k)-\Lambda_{V_2,q_2}^\Gamma(k)\|_{\tilde{H}^{\half}(\Gamma) \rightarrow H^{-\half}(\Gamma)}$ and applying the previous estimates leads to 
\begin{equation*}
\Big|\int_{\Omega'} (b_2-b_1) \tilde u_{1} \tilde u_{2} \, d x\Big|
\leq \delta \|\tilde u_1\|_{H^{\half}(\p\Omega)}\|\tilde u_2\|_{H^{\half}(\p\Omega)}\leq  C \delta e^{Ck}\epsilon^{-2\nu}\|u_1\|_{L^2(\Omega)}\|u_2\|_{L^2(\Omega)}.
 \end{equation*}

We extend $b_j$ by zero to $\R^n$.
Now we seek to apply the previous steps to estimate 
\begin{align*}
|\mathcal F(b_2-b_1)(r\omega)|
&=\left|\int_{\Omega'} (b_2-b_1) e^{-ir\omega\cdot x}dx \right|
\end{align*}
for any $|r|\leq 2\tau$ and $\omega\in \mathbb S^{n-1}$.
Notice that 
\begin{align*}
e^{-ir\omega\cdot x}
&=u_1u_2-e^{-ir\omega\cdot x}(\psi_1+\psi_2+\psi_1\psi_2)
\\
&=-e^{-ir\omega\cdot x}(\psi_1+\psi_2+\psi_1\psi_2)
+(u_1-\tilde u_1)u_2+(u_2-\tilde u_2)\tilde u_1+\tilde u_1\tilde u_2.
\end{align*}
Thus, using the Runge approximation bounds again and invoking the estimates for the functions $u_j$ and $\psi_j$, we obtain
\begin{align*}
|\mathcal F(b_2-b_1)(r\omega)|
&\leq Ck^2B\left(\frac{k^2B}
{\tau}+\epsilon \tau^2e^{2\tau}
\right)+C\delta e^{Ck}\epsilon^{-2\nu} e^{2\tau}.
\end{align*}

In order to estimate $\|b_2-b_1\|_{H^{-1}(\Omega)}$, we notice that for any $\rho<2\tau$
\begin{align*}
\|b_2-b_1\|_{H^{-1}(\Omega)}^2
&=\int_{\R^n}|\F(b_2-b_1)(\zeta)|^2(1+|\zeta|^2)^{-1}d\zeta
\\
&\leq  \int_{|\zeta|<\rho}|\F(b_2-b_1)(\zeta)|^2(1+|\zeta|^2)^{-1}d\zeta
+\frac{1}{1+\rho^2}\|b_1-b_2\|_{L^2(\Omega)}^2
\\
&\leq C\rho^{n-2}\left(\frac{(k^2B)^4}
{\tau^2}+(k^2B)^2\epsilon^2 e^{5\tau}
+\delta^2 e^{Ck}\epsilon^{-4\nu} e^{5\tau}\right)+C\rho^{-2}(k^2 B)^2.
\end{align*}

Choosing $\rho=\left(\frac{\tau}{k^2B}\right)^{\frac 2n}$
and $\epsilon=\delta^{\frac{1}{1+2\nu}}$ yields
\begin{align*}
\|b_2-b_1\|_{H^{-1}(\Omega)}^2
&\leq C\left((k^2B)^{2+\frac 4 n} \tau^{-\frac 4n}+ e^{C\tau}e^{Ck}\delta^{\frac{2}{1+2\nu}}\right).
\end{align*}
Now we assume $\delta<1$, recall that $\tau \geq \max\{C_0k^2 B, 1\}$ and choose 
\begin{align*}
C\tau=CC_0k^{n+3}B-\left(\frac{1}{1+2\nu}\right)\log \delta,
\end{align*}
which results in 
\begin{align*}
\|b_2-b_1\|_{H^{-1}(\Omega)}^2
&\leq C\frac{1}{(k+k^{-(n+2)}|\log\delta|)^{\frac 4 n}}+e^{Ck^{n+3}}\delta^{\frac{1}{1+2\nu}},
\end{align*}
where the constant $C>0$ now includes the $B$-dependence.
Applying Young's inequality we can estimate the last term as follows:
\begin{align*}
e^{Ck^{n+3}} \delta^{\frac{1}{1+2\nu}}
\leq C \left( e^{Ck^{n+3}}\delta^{2}
+\frac{\delta^{\frac{2}{3+8\nu}}}{k^{\frac 4 n}}\right).
\end{align*}
Taking into account that  ${\delta^{\alpha}}{k^{-\frac 4 n}}\leq (k+\frac n 4\alpha |\log\delta|)^{-\frac 4 n}$,  then
\begin{align*}
\|b_2-b_1\|_{H^{-1}(\Omega)}
&\leq C\left(\frac{1}{(k+k^{-(n+2)}|\log\delta|)^{\frac 2 n}}+e^{Ck^{n+3}}\delta\right).
\end{align*}

In order to infer the desired result, we finally notice that
\begin{align*}
\frac{1}{k+k^{-(n+2)}|\log\delta|}
\leq C\frac{1}{k+|\log\delta|^{\frac{1}{n+3}}}.
\end{align*}
Indeed, applying again Young's inequality, we have
\begin{align*}
|\log\delta|^{\frac{1}{n+3}}\leq \frac{1}{\min\{p,q\}}\left({(k^{-\frac{n+2}{n+3}}|\log\delta|^{\frac{1}{n+3}})^p}+{k^{\frac{n+2}{n+3}q}}\right).
\end{align*}
Choosing $p=n+3$ and $q=\frac{n+3}{n+2}$, the previous claim follows.
\end{proof}

\section{Improved Carleman Estimates and Three Balls Inequalities in the Presence of Convexity}
\label{sec:improved_results_convexity}

Last but not least, in this section we show how in the presence of convexity of the domains the Runge approximation results can be improved. This provides the Runge approximation counterpart to the improved stability estimates for unique continuation. Since we need quantitative unique continuation estimates in the natural trace spaces, we also provide the relevant Carleman estimates. In other functional settings similar results had been proved earlier in the literature, see for instance \cite{HI04,IS07}. In order to illustrate the effect, we consider the geometric setting of concentric balls but remark that this could also be extended to other convex geometries.

\subsection{Improved unique continuation results}
\label{sec:UCP_improve}

We seek to deduce improved unique continuation estimates in $k$.
To this end, we first derive a Carleman estimate with improvements in $k$ for  the model case of the acoustic equation without potential
\begin{align*}
(\D + k^2 q)u & = 0 \; \mbox{ in } \Omega.
\end{align*}
 Here we differ slightly from the argument by Isakov and use ideas from results on excluding embedded eigenvalues instead (see for instance \cite{KT06}).

\begin{prop}
\label{prop:Carl_eigenval}
Let $u: \R^n \rightarrow \R$ be  compactly supported in $B_2 \backslash \overline{B_1} \subset \R^n\backslash \{0\}$ and solve  
\begin{align}\label{eq:eqCarl}
(\D + k^2 q) u & = f + \sum_{j=1}^n\p_j F^j\; \mbox{ in } \R^n,
\end{align}
where $q$ satisfies  \textnormal{(\hyperref[assq2]{ii'})} in $\R^n$ and $f,\,F^j \in L^2(\R^n)$ with  $\supp f,  \,\supp F^j\subset B_2\backslash\overline{B_1}$. Let $\phi(x) := \tau \log(|x|) $. Then, there exists $\tau_0>0$ such that for any $\tau\geq \tau_0$, there is  a constant $C>0$ depending on $n$ and $\kappa$ such that
\begin{align}\label{eq:Carleman}
\begin{split}
&\tau \| e^{\phi}  u\|_{L^2(\R^n)}
 +  \|e^{\phi} |x| \nabla u\|_{L^2(\R^n)}
 + \tau^{\half} k \|
 {q^{\half}|x|}e^{ \phi} u\|_{L^2(\R^n)}\\
 &
 \quad \leq C  \big( \|e^{ \phi} |x|^2 f\|_{L^2(\R^n)} + {\max\{\tau,k\}} \sum_{j=1}^n\|e^{\phi}|x| F^j\|_{L^2(\R^n)}\big).
 \end{split}
\end{align}
\end{prop}

\begin{rmk}
As is common in stability improvement results, the main feature of the Carleman estimate from Proposition ~\ref{prop:Carl_eigenval} is that the frequency is included in the right hand side (the main part of the operator) and that there is an improvement depending on $k$ on the left hand side of the estimate. Such ideas also hold for more general operators (see, for instance, \cite{KT06} or \cite{I19}).
\end{rmk}

\begin{proof}
We argue in three steps: First, we pass to conformal polar coordinates, then we invoke a splitting strategy in which we split the conjugated equation into an elliptic and a subelliptic contribution. For these we separately deduce the corresponding estimates. Finally, we combine these two estimates into the desired overall bound.\\

\emph{Step 1: Coordinate transformation.}
We pass to conformal polar coordinates, i.e. we set $x = \psi (t,\theta)$ for \begin{align*}
\psi:  \R \times {\mathbb{S}^{n-1}} &\to \R^n\backslash \{0\},
\\ (t,\theta) &\mapsto e^{t} \theta.
\end{align*}
 For any $\varphi\in L^1(\R^n, |x|^{-n}dx)$ we obtain with the area formula 
\begin{align*}
\int_\R \int_{\mathbb{S}^{n-1}} \varphi \circ \psi (t,\theta) d \mathcal H^{n-1}(\theta) dt = \int_{\R^n}\frac{1}{|x|^{n}}\varphi(x) dx,
\end{align*} where $\mathcal H^{n-1}$ denotes the $n-1$ dimensional Hausdorff measure on $\mathbb{S}^{n-1}$. Thus, at least formally, $d\mathcal H^{n-1}(\theta)dt=d\theta dt = |x|^{-n}dx$. 

A standard calculation shows that in the new coordinates
\begin{align*}
(|x|^2\D u) \circ \psi = \big(\p_t^2 +(n-2)\p_t + \D_{\mathbb{S}^{n-1}}\big)(u \circ \psi ).
\end{align*}
We can discard the first order term by conjugating the operator above with $|x|^{-\frac{n-2}{2}}$, that is $e^{-\frac{n-2}{2}t}$ in the new coordinates. Therefore  \eqref{eq:eqCarl} becomes the following equation for  $\tilde u(t,\theta)=e^{\frac{n-2}{2}t} u \circ \psi(t,\theta)$:
\begin{align}\label{eq:eqfortildeu}
\left(\p_t^2 + \D_{\mathbb \mathbb S^{n-1}}+k^2e^{2t}\tilde q-c_n\right)\tilde u= \tilde f+\p_t \tilde F^t+\mbox{div}_{\mathbb  S^{n-1}} \tilde F^\theta,
\end{align}
where 
\begin{align*}
\tilde q(t,\theta)
&=q\circ \psi (t,\theta), \qquad
c_n=\left(\frac{n-2}{2}\right)^2,
\\
\tilde f(t,\theta)
&=e^{2t}e^{\frac{n-2}{2}t}f\circ\psi(t, \theta)+\left(\frac{n}{2}-1\right) \tilde F^t(t,\theta),
\\ 
\tilde F^t(t,\theta)
&=e^{\frac n 2 t}\left(\sum_{j=2}^n \theta_{j-1} F^j\circ\psi(t, \theta)\pm \Big({1-\sum_{i=1}^{n-1}\theta_i^2}\Big)^{\frac 1 2} F^1\circ\psi(t, \theta)\right),
\\
\tilde F^{\theta_i}(t,\theta)
&=e^{\frac n 2 t}F^{i+1}\circ\psi(t, \theta)-\theta_i \tilde F^t(t,\theta), \quad i\in\{1, \dots, n-1\},
\\
\tilde F^{\theta}(t,\theta)
&=\big(\tilde F^{\theta_1}(t,\theta), \dots, \tilde F^{\theta_{n-1}}(t,\theta)\big),
\\
\theta_i
&=\frac{x_{i+1}}{|x|}, \quad  i\in\{1, \dots, n-1\}.
\end{align*}
For ease of notation and later reference, we denote the operator on the left hand side of \eqref{eq:eqfortildeu} by $L$. In addition, for some given weight $\Phi=\Phi(t)$, we denote by $L_\Phi$ the conjugated operator given by
\begin{align}\label{eq:LPhi}
L_{\Phi}= e^\Phi L e^{-\Phi}= \p_t^2 +\D_{\mathbb S^{n-1}}-2\Phi' \p_t+k^2e^{2t} \tilde q +{\Phi'^2-\Phi''}- c_n.
\end{align}

\emph{Step 2: Splitting strategy.} 
In order to deal with the divergence contributions, we use a splitting strategy and set $u= u_1 +  u_2$, where  $\tilde u_1(t,\theta):=e^{\frac{n-2}{2}t} u_1 \circ \psi(t,\theta)$ is a weak solution to 
\begin{align*}
\left(L-D \max\{\tau^2, k^2 e^{2t} \tilde q \} \right)\tilde u_1= \tilde f+\p_t \tilde F^t+\mbox{div}_{\mathbb S^{n-1}} \tilde F^\theta,
\end{align*}
for $D>0$ large. A solution to this exists by Lax-Milgram. Indeed, this follows by considering the bilinear form
\begin{align*}
B(h_1, h_2)= \int_{\mathbb S^{n-1}}\!\int_\R \Big( \p_t h_1\p_t h_2+\nabla_{\mathbb S^{n-1}} h_1\cdot\nabla_{\mathbb S^{n-1}} h_2+ bh_1h_2 \Big)dtd\theta
\end{align*}
with $b=D \max\{\tau^2, k^2 e^{2t}\tilde q\}-k^2 e^{2t}\tilde q+c_n>0$.
An application of the Lax-Milgram theorem with this bilinear form then yields a solution $\tilde u_1\in H^1(\R\times\mathbb S^{n-1})$ and associated energy bounds in terms of the $L^2$ norms of $\tilde{f}$, $\tilde{F}^t, \tilde{F}^{\theta}$.

The equation for the function $\tilde u_2(t,\theta)=e^{\frac{n-2}{2}t} u_2 \circ \psi(t,\theta)$ is determined by considering the difference of $\tilde u-\tilde u_1$.\\

\emph{Step 2a: Energy estimates for $u_1$.}
We seek to complement the existence result for $\tilde{u}_1$ with exponentially weighted energy estimates. Since the support of $\tilde{u}_1$ is in general not bounded, we first consider the conjugated equation with a truncated weight. Energy estimates which are uniform in the truncation parameter and a limiting argument then allow us to pass to the desired weight.
To this end, we consider a smooth weight $\Phi_R$ for $R\geq 2$ such that $\Phi_R(t)=(1+\tau)t$ for $t\leq R$ and $\Phi_R(t)=(1+\tau)\frac{3R}{2}$ for $t\geq 2R$. In addition, $\Phi'_R\leq {1+\tau}$ and $\Phi_R''\leq \frac{1+\tau}{R}$ in their corresponding supports.
Let $w_R=e^{\Phi_R}\tilde u_1$, then  it satisfies the  equation
\begin{align*}
\big(L_{\Phi_R}-D\max\{\tau^2, k^2 e^{2t}\tilde q\}\big)w_R
=e^{\Phi_R}\left(\tilde f+\p_t \tilde F^t+\mbox{div}_{\mathbb S^{n-1}} \tilde F^\theta\right),
\end{align*}
where $L_{\Phi_R}$ is given by \eqref{eq:LPhi}.
Testing the equation for $w_R$ with itself yields
\begin{align*}
&\|\p_t w_R\|_{L^2(\R^n\times\mathbb S^{n-1})}^2
+\|\nabla_{\mathbb S^{n-1}}w_R\|_{L^2(\R\times\mathbb S^{n-1})}^2
\\
&\qquad +\int_{\mathbb S^{n-1}}\!\int_\R \left(D\max\{\tau^2,k^2 e^{2t}\tilde q \}-k^2 e^{2t}\tilde q+c_n\right) w_R^2 dtd\theta
\\
& \qquad+\int_{\mathbb S^{n-1}}\!\int_\R (\Phi_R''-\Phi_R'^2)w_R^2  dtd\theta
+2\int_{\mathbb S^{n-1}}\!\int_\R \Phi_R'w_R\p_t w_R dtd\theta
\\
&=-\int_{\mathbb S^{n-1}}\!\int_\R e^{\Phi_R}\left(\tilde f+\p_t \tilde F^t+\mbox{div}_{\mathbb S^{n-1}} \tilde F^\theta\right)w_R dtd\theta.
\end{align*}
Applying integration by parts and Young's inequality, together with the fact that $\supp \tilde f, \, \supp \tilde F^j\subset (0,\log 2)\times\mathbb S^{n-1}=:I \times \mathbb{S}^{n-1}$, we  obtain the following estimates for $\tau>1$
\begin{align*}
\left|\int_{\mathbb S^{n-1}}\!\int_\R (\Phi_R''-\Phi_R'^2)w_R^2  dtd\theta\right|
&\leq  6\tau^2 \|w_R\|_{L^2(\R\times\mathbb S^{n-1})}^2,
\\
\left|2\int_{\mathbb S^{n-1}}\!\int_\R \Phi_R'w_R\p_t w_R dtd\theta\right|
&\leq 16\tau^2\|w_R\|_{L^2(\R\times\mathbb S^{n-1})}^2 +\frac{1}{4}\|\p_t w_R\|_{L^2(\R\times\mathbb S^{n-1})}^2 ,
\\
\left|\int_{\mathbb S^{n-1}}\!\int_\R e^{\Phi_R} \tilde fw_R dtd\theta\right|
&\leq \frac{\kappa}{\max\{\tau^2, k^2\}}\|e^{\Phi_R}\tilde f\|_{L^2(\R\times\mathbb S^{n-1})}^2
+ \frac 1 4 \frac{1}{\kappa}\max\{\tau^2, k^2 \}\|w_R\|_{L^2(I\times\mathbb S^{n-1})}^2
\\
&\leq  \frac{\kappa}{\max\{\tau^2, k^2\}}\|e^{(1+\tau)t}\tilde f\|_{L^2(\R\times\mathbb S^{n-1})}^2\\
& \quad 
+ \frac 1 4 \|\max\{\tau^2, k^2 e^{2t}\tilde q\}^{\half}w_R\|_{L^2(I\times\mathbb S^{n-1})}^2
\\
&\leq  \frac{\kappa}{\max\{\tau^2, k^2\}}\|e^{(1+\tau)t}\tilde f\|_{L^2(\R\times\mathbb S^{n-1})}^2\\
& \quad 
+ \frac 1 4 \|\max\{\tau^2, k^2 e^{2t}\tilde q\}^{\half}w_R\|_{L^2(\R\times\mathbb S^{n-1})}^2,
\\
\left|\int_{\mathbb S^{n-1}}\!\int_\R  e^{\Phi_R}\p_t \tilde F^tw_R dtd\theta\right|
& \leq C\|e^{(1+\tau)t}\tilde F^t\|_{L^2(\R\times\mathbb S^{n-1})}^2+ \frac{\tau^2}{4}\|w_R\|_{L^2(\R\times\mathbb S^{n-1})}^2
+\frac 1 4\|\p_tw_R\|_{L^2(\R\times\mathbb S^{n-1})}^2,
\\
\left|\int_{\mathbb S^{n-1}}\!\int_\R e^{\Phi_R} \mbox{div}_{\mathbb S^{n-1}} \tilde F^\theta w_R dtd\theta\right|
& \leq \|e^{(1+\tau)t}\tilde F^\theta\|_{L^2(\R\times\mathbb S^{n-1})}^2
+\frac 1 4\|\nabla_{\mathbb S^{n-1}}w_R\|_{L^2(\R\times\mathbb S^{n-1})}^2.
\end{align*}
Absorbing the terms with $w_R$ and the non-positive terms  for $D$ sufficiently large, we obtain
\begin{align*}
&\tau\|w_R\|_{L^2(\R\times\mathbb S^{n-1})}
+\|\p_t w_R\|_{L^2(\R\times\mathbb S^{n-1})}
+\|\nabla_{\mathbb S^{n-1}}w_R\|_{L^2(\R\times\mathbb S^{n-1})}
+k\| e^t\tilde q^{\half}w_R\|_{L^2(\R\times\mathbb S^{n-1})}
\\
&\quad\leq   C\left( \frac{1}{\max\{\tau, k\}}\|e^{(1+\tau)t}\tilde f\|_{L^2(\R\times\mathbb S^{n-1})}
+\|e^{(1+\tau)t}\tilde F^t\|_{L^2(\R\times\mathbb S^{n-1})}
+ \|e^{(1+\tau)t}\tilde F^\theta\|_{L^2(\R\times\mathbb S^{n-1})}\right).
\end{align*}
Notice that the right hand side is finite and does not depend on $R$, so taking $R\to \infty$, we obtain similar estimates for $w=e^{(1+\tau)t} \tilde u_1$. Multiplying the whole expression by $\max\{\tau, k\}$ 
and returning to the original coordinates we arrive at
\begin{align}\label{eq:estimateu1i}
\begin{split}
&\tau^2\|e^\phi u_1\|_{L^2(\R^n)} 
+\max\{\tau, k\}\|e^\phi |x|\nabla u_1\|_{L^2(\R^n)}
+\max\{\tau,k\}k\|q^{\half} e^\phi |x| u_1\|_{L^2(\R^n)}
\\
&\qquad\leq   C\Big(\|e^\phi|x|^2f\|_{L^2(\R^n)}+
\max\{\tau,k\}\sum_{j=1}^n\|e^\phi |x|F^j\|_{L^2(\R^n)}\Big).
\end{split}
\end{align}

Arguing similarly for $\tilde \Phi_R$ with $\tilde \Phi_R=(2+\tau)t$ if $t\leq R$, we also deduce
\begin{align}\label{eq:estimateu1ii}
\begin{split}
k^2\|e^\phi|x|^2q^{\half}u_1\|_{L^2(\R^n)}
\leq   C\Big(\|e^\phi|x|^3f\|_{L^2(\R^n)}+
\max\{\tau,k\}\sum_{j=1}^n\|e^\phi |x|^2F^j\|_{L^2(\R^n)}\Big).
\end{split}
\end{align}
Combining \eqref{eq:estimateu1i}-\eqref{eq:estimateu1ii} and exploiting again the compact support of $f$ and $F^j$ and \textnormal{(\hyperref[assq2]{ii'})}, we infer that 
\begin{align}\label{eq:estimateu2RHS}
\begin{split}
\|D\max\{\tau^2, {k^2|x|^2 q}\} e^\phi u_1\|_{L^2(\R^n)}
&\leq D \tau^2\|e^\phi u_1\|_{L^2(\R^n)}
+D\kappa^{\half} k^2\||x|^2 q^{\half}  e^\phi u_1\|_{L^2(\R^n)}
\\
&\leq C \Big(\|e^\phi|x|^2f\|_{L^2(\R^n)}+
\max\{\tau,k\}\sum_{j=1}^n\|e^\phi |x|F^j\|_{L^2(\R^n)}\Big),
\end{split}
\end{align} 
where now $C$ also depends on $\kappa$.\\

\emph{Step 2b: Carleman estimates for $u_2$.}
We now consider the estimate for $u_2$. To this end, we note that $\tilde u_2$ solves the equation
\begin{align*}
L \tilde u_2 = D \max\{\tau^2, k^2 e^{2t}\tilde q\} \tilde u_1 \; \mbox{ in } \R \times \mathbb{S}^{n-1}.
\end{align*}

We now carry out the conjugation with $e^{\Phi}$ for $\Phi(t)=(1+\tau) t$ and split the operator $L_{\Phi}$ given in \eqref{eq:LPhi}  into its symmetric and antisymmetric parts (with respect to the $L^2(\R \times {\mathbb{S}^{n-1}})$ scalar product)
\begin{align*}
S_{\Phi} = \p_t^2 + \D_{\mathbb S^{n-1}} + k^2e^{2t} \tilde q +{(1+\tau)^2}- c_n,\quad
A_{\Phi} = -2(1+\tau)\p_t.
\end{align*}
Let us set $v=e^{\Phi} \tilde u_2$. Expanding the right hand side of the last equality, we obtain
\begin{align*}
\|L_{\Phi} v\|_{L^2(\R \times \mathbb S^{n-1})}^2
= \|S_{\Phi} v\|_{L^2(\R \times \mathbb S^{n-1})}^2 + \|A_{\Phi} v\|_{L^2(\R \times \mathbb S^{n-1})}^2 + ([S_{\Phi},A_{\Phi}]v ,v)_{L^2(\R \times \mathbb S^{n-1})}. 
\end{align*}
In addition, by the definition of  $L_\Phi$   and $v$,
\begin{align*}
\|L_{\Phi} v \|_{L^2(\R \times \mathbb S^{n-1})}=\|e^{\Phi}L\tilde u_2\|_{L^2(\R \times \mathbb S^{n-1})}= \|D\max\{\tau^2, {k^2|x|^2 q}\} e^\phi u_1\|_{L^2(\R^n)}.
\end{align*}

We begin with a lower bound on the commutator. We calculate 
\[
[S_{\Phi},A_{\Phi}]v =[ e^{2t } k^2\tilde q, A_{\Phi}]v=2 (1+\tau)\p_t(k^2e^{2t}\tilde q)  v.
\]
As $\p_t \tilde q= (\nabla q\cdot x) \circ \psi \ge 0 $ and $\tilde{q}> 0$ by assumption, we thus find after returning to the standard coordinates
\[
([S_{\Phi},A_{\Phi}]v,v)_{L^2(\R \times \mathbb S^{n-1})}\ge 4(1+\tau)(e^{2t}k^2\tilde q v,v)_{L^2(\R \times \mathbb S^{n-1})}=4(1+\tau)k^2\int_{\R^n}q|x|^2e^{2\phi}u^2_2 \,dx.
\] Therefore, we conclude
\begin{equation}
    \label{eq:1}
    4(1+\tau)k^2\int_{\R^n}q|x|^2e^{2\phi}u^2_2 \,dx\leq   \|D\max\{\tau^2, {k^2|x|^2 q}\} e^\phi u_1\|_{L^2(\R^n)}^2.
\end{equation}

Now we seek to estimate $\|v\|_{L^2(\R\times\mathbb S^{n-1})}$ in terms of $\|A_{\Phi}v\|_{L^2(\R \times \mathbb{S}^{n-1})}=2(1+\tau)\|\p_tv\|_{L^2(\R \times \mathbb{S}^{n-1})}$. Using the compact support of $\tilde{u}$, we can apply the Poincar\'e inequality to the function $e^\Phi \tilde u(\,\cdot\,, \theta)$ for almost every $\theta\in \Sn$ as follows
\begin{align*}
\|v\|_{L^2(\R \times \mathbb{S}^{n-1})}
&\leq \|e^\Phi \tilde u\|_{L^2(\R \times \mathbb{S}^{n-1})}+\|e^\Phi \tilde u_1\|_{L^2(\R \times \mathbb{S}^{n-1})}
\\&\leq C\big(\|\p_t(e^\Phi \tilde u)\|_{L^2(\R \times \mathbb{S}^{n-1})}+\|e^\Phi \tilde u_1\|_{L^2(\R \times \mathbb{S}^{n-1})}\big)
\\&\leq C\big(\|\p_tv\|_{L^2(\R \times \mathbb{S}^{n-1})}+\|\p_t(e^\Phi \tilde u_1)\|_{L^2(\R \times \mathbb{S}^{n-1})}+\|e^\Phi \tilde u_1\|_{L^2(\R \times \mathbb{S}^{n-1})}\big)
\\&\leq C\big(\|\p_tv\|_{L^2(\R \times \mathbb{S}^{n-1})}+\|e^\Phi \p_t\tilde u_1\|_{L^2(\R \times \mathbb{S}^{n-1})}+\tau\|e^\Phi \tilde u_1\|_{L^2(\R \times \mathbb{S}^{n-1})}\big),
\end{align*}
where $C$ depends on $n$ (and the support of $\tilde u$).
Multiplying the whole inequality by $(1+\tau)$ we can write
\begin{align*}
(1+\tau)\|v\|_{L^2(\R \times \mathbb{S}^{n-1})}
&\leq C\big(\|A_\Phi v\|_{L^2(\R \times \mathbb{S}^{n-1})}+\tau\|e^\Phi \p_t\tilde u_1\|_{L^2(\R \times \mathbb{S}^{n-1})}+\tau^2\|e^\Phi \tilde u_1\|_{L^2(\R \times \mathbb{S}^{n-1})}\big).
\end{align*}
Returning to Euclidean coordinates  yields
\begin{align}
\begin{split}
    \label{eq:2}
    (1+\tau)\| e^{\phi} u_2 \|_{L^2(\R^n)}
    &\leq  C \big(\|D\max\{\tau^2, {k^2|x|^2 q}\} e^\phi u_1\|_{L^2(\R^n)}
     + \tau^2 \|e^{\phi} u_1\|_{L^2(\R^n)}\\
&\qquad \quad +\tau\|e^{\phi} |x| \nabla u_1 \|_{L^2(\R^n)}\big).
\end{split}
\end{align}

Lastly, we deduce a gradient bound on $\tilde u_2$ using the symmetric part of the operator.  Testing $S_\Phi v$ with $v$ and integrating by parts  we obtain 
\begin{align*}
(S_\Phi v,v)_{L^2(\R\times \mathbb{S}^{n-1})}
&=-\|\p_t v\|_{L^2(\R\times \mathbb{S}^{n-1})}^2
-\|\nabla_{\mathbb S^{n-1}}v\|_{L^2(\R\times \mathbb{S}^{n-1})}^2
\\
&\quad+k^2(e^{2t}\tilde q v,v)_{L^2(\R\times \mathbb{S}^{n-1})}
+\left((1+\tau)^2-c_n\right)\|v\|^2_{L^2(\R\times \mathbb{S}^{n-1})}.
\end{align*}
Therefore,
\begin{align*}
\|\p_t v\|_{L^2(\R\times \mathbb{S}^{n-1})}^2
+\|\nabla_{\mathbb S^{n-1}}v\|_{L^2(\R\times \mathbb{S}^{n-1})}^2
&\leq 
\|S_\Phi v\|_{L^2(\R\times \mathbb{S}^{n-1})}^2
+k^2(e^{2t}\tilde q v,v)_{L^2(\R\times \mathbb{S}^{n-1})}\\
& \quad +2(1+\tau)^2\|v\|^2_{L^2(\R\times \mathbb{S}^{n-1})}.
\end{align*}
Returning to the original coordinates and using   \eqref{eq:1}-\eqref{eq:2} to estimate the right hand side yields
\begin{align}\label{eq:3}
\begin{split}
\|e^\phi|x|\nabla u_2\|_{L^2(\R^n)}
 &\leq  C \big(\|D\max\{\tau^2, {k^2|x|^2 q}\} e^\phi u_1\|_{L^2(\R^n)}
     + \tau^2 \|e^{\phi} u_1\|_{L^2(\R^n)}\\
&\qquad \quad + \tau\|e^{\phi} |x| \nabla u_1 \|_{L^2(\R^n)}\big).\end{split}
\end{align}
Finally, combining \eqref{eq:1},\eqref{eq:2} and\eqref{eq:3} with \eqref{eq:estimateu1i} and \eqref{eq:estimateu2RHS}, we obtain for $\tau>1$
\begin{align}\label{eq:estimateu2}
\begin{split}
&\tau \| e^{\phi}  u_2\|_{L^2(\R^n)}
 +  \|e^{\phi} |x| \nabla u_2\|_{L^2(\R^n)}
 + \tau^{\half} k \|
 {q^{\half}|x|}e^{ \phi} u_2\|_{L^2(\R^n)}
\\&\quad\leq C\Big(\|e^\phi|x|^2f\|_{L^2(\R^n)}+
\max\{\tau,k\}\sum_{j=1}^n\|e^\phi |x|F^j\|_{L^2(\R^n)}\Big).
\end{split}
\end{align}

\emph{Step 3: Conclusion.} 
The final estimate \eqref{eq:Carleman} follows  by an application of the triangle  inequality and
the estimates  \eqref{eq:estimateu1i} and \eqref{eq:estimateu2}  for $u_1$ and $u_2$, respectively.  \end{proof}

Next, using the previous Carleman estimate, we deduce a quantitative unique continuation result which does not suffer from the losses in $k$.

\begin{thm}
\label{prop:UCP_improved_ink}
Let $V$ and $q$ be as in \textnormal{(\hyperref[assV]{i})}-\textnormal{(\hyperref[assq2]{ii'})} in $\Omega=B_2 \backslash \overline{B_1}$. Let $u\in H^1(\Omega)$ be a solution to 
\begin{align}\label{eq:eqinOm1}
\begin{split}
(\Delta+k^2q+V)u&=0 \,  \mbox{ in }  \, \Omega,
\end{split}
\end{align} 
and let $M$, $\eta$ be such that
\begin{align*}
\|u\|_{H^1(\Omega)} &\leq M,\\
\|u\|_{H^{\half}(\p B_2)}+\|\p_{\nu} u \|_{H^{-\half}(\p B_2)}&\leq \eta.
\end{align*}
Assume further that $0<{k^3} \eta \leq M$.
Then there exist  a parameter $\mu\in(0,1)$ and a constant $C>1$ depending on $n, \|V\|_{L^\infty(\Omega)}, \kappa$ and $\|q\|_{C^1(\Omega)}$ (but not on $k$) such that 
\begin{align}\label{eq:QUCPboundary_improv}
\|u\|_{L^2(\Omega)}\leq C \left|\log\left(\frac{k^3\eta}{M}\right)\right|^{-\mu}M.
\end{align}
In addition, if $G=B_2 \backslash B_{1+\delta}$ for some $\delta \in (0,1)$, then there exist a parameter $\nu\in(0,1)$ and a constant $C>1$ (depending on $n, \delta, \|V\|_{L^\infty(\Omega)}, \kappa$ and $\|q\|_{C^1(\Omega)}$ but not on $k$) such that
\begin{align}\label{eq:QUCPinterior_improv}
\|u\|_{L^2(G)}\leq  C \left({k^3\eta}\right)^\nu M^{1-\nu}.
\end{align}
\end{thm}

We will prove Theorem ~\ref{prop:UCP_improved_ink} in several steps. First we prove a corresponding propagation of smallness result from the interior for divergence form equations. Combined with an extension argument this will then lead to the desired claim of Theorem ~\ref{prop:UCP_improved_ink}.

\begin{prop}
\label{prop:div_form_rhs}
Let $V$ and $q$ be as in \textnormal{(\hyperref[assV]{i})}-\textnormal{(\hyperref[assq2]{ii'})} in $\Omega=B_2 \backslash \overline{B_1}$.
Let $u \in H^1(\Omega)$ with $\supp u\subset B_2 \backslash B_1$ be a solution to 
\begin{align*}
(\D + k^2 q  + V ) u & = f+\sum_{j=1}^n\p_j F^j \; \mbox{ in } \; \Omega,
\end{align*}
where $f,\, F^j\in L^2(\Omega)$ with $\supp f,\, \supp F^j\subset B_2\backslash B_1$. Let $M, \eta>0$ be such that
\begin{align*}
\|u\|_{H^1(\Omega)} &\leq M,\\
\|f\|_{L^2(\Omega)}+\sum_{j=1}^n \|F^j\|_{L^2(\Omega)} &\leq \eta. 
\end{align*}
Assume further that $0<k \eta \leq M$.
Then there exist $\mu \in (0,1)$ and $C>1$ (depending on n, $\|V\|_{L^\infty(\Omega)}$ and $\kappa$) such that
\begin{align}\label{eq:QUCPboundary_impr}
\|u\|_{L^2(\Omega)}
 \leq C \left|\log\left(\frac{k\eta}{M}\right)\right|^{-\mu}M.
\end{align}
In addition, if $G=B_2 \backslash \overline{B_{1+\delta}}$ for some $\delta \in (0,1)$, then there exist a parameter $\nu\in(0,1)$ and a constant $C>1$ (depending on $n, \delta$, $\|V\|_{L^\infty(\Omega)}$ and $\kappa$) such that
\begin{align}\label{eq:QUCPinterior_impr}
\|u\|_{L^2(G)}\leq C (k\eta)^\nu M^{1-\nu}.
\end{align}
\end{prop}

\begin{rmk}
We emphasise that unlike in Proposition ~\ref{prop:Carl_eigenval}, in Proposition ~\ref{prop:div_form_rhs}, we are not assuming that $u$ and the functions $f$ and $F^j$ vanish on the interior boundary $\p B_1$.
\end{rmk}

\begin{proof}[Proof of Proposition ~\ref{prop:div_form_rhs}]
We use the Carleman inequality from Proposition ~\ref{prop:Carl_eigenval} in combination with the Sobolev embedding theorem and an optimization argument. We argue in two steps, first proving \eqref{eq:QUCPinterior_impr} and then using this to prove \eqref{eq:QUCPboundary_impr} .\\

\emph{Step 1: Proof of \eqref{eq:QUCPinterior_impr}}. We apply the Carleman estimate from Proposition ~\ref{prop:Carl_eigenval} to the function $w:= u \chi$, where $\chi $ is a smooth cut-off function which is equal to one on $B_2 \backslash B_{1+\sfrac{\delta}{2}}$, vanishes on $B_{1+\sfrac{\delta}{4}}$ and satisfies $|\nabla\chi|\leq C\delta^{-1}$, $|\Delta\chi|\leq C\delta^{-2}$. 
The function $w$ thus is compactly supported in $B_2\backslash \overline{B_1}$ and solves the equation
\begin{align}\label{eq:eqw}
(\D + k^2 q) w & = -Vw +g +  \sum_{j=1}^n\p_j G^j \quad \mbox{ in } B_2 \backslash \overline{B_1},
\end{align}
where 
\begin{align*}
g&=u \D \chi  + 2 \nabla u \cdot \nabla \chi+\chi f - \sum_{j=1}^n(\p_j \chi)F^j,\
\quad G^j=\chi F^j.
\end{align*}
We now seek to apply Proposition \ref{prop:Carl_eigenval}. To this end, we first extend $q$ to $\R^n$ such that \textnormal{(\hyperref[assq2]{ii'})} remains true. To this end, we first notice that  \eqref{eq:eqw} only depends on the value of $q$  in some domain $\Omega'=B_{2-\epsilon}\backslash B_{1+\sfrac{\delta}{4}}\Subset\Omega$, where $\epsilon\in(0,\frac 1 2)$ depends on the support of $u$. Let $\xi$ be a radial smooth function supported in $\Omega$ and such that $\xi=1$ in $\Omega'$, $\p_r\xi\geq0$ in the bounded component of $\Omega\backslash\Omega'$ and 
$\p_r\xi\leq0$ otherwise.
Now we consider the function $\tilde q=\xi q + (1-\xi)(\kappa^{-1}\mathbb{1}_{B_{\sfrac 3 2}}+\kappa\mathbb{1}_{\R^n\backslash B_{\sfrac{3}{2}}})$, which coincides with $q$ in $\Omega'$. It is clear that $\tilde q\in C^1(\R^n)$ and $\kappa^{-1}\leq \tilde q\leq \kappa$ in $\R^n$.
Finally, since $\nabla q\cdot x\geq 0$ in $\Omega$, $\p_r\xi(q-\kappa^{-1})\geq 0$ in $(\Omega\backslash\Omega')\cap B_{\sfrac{3}{2}}$  and  $\p_r\xi(q-\kappa)\geq 0$ in $(\Omega\backslash\Omega')\cap (\R^n\backslash B_{\sfrac{3}{2}})$  we deduce $\nabla \tilde q\cdot x\geq 0$ in $\R^n$.

Therefore, invoking Proposition ~\ref{prop:Carl_eigenval}, we obtain for $\tau>\tau_0$
\begin{align}\label{eq:proofcom}
\tau \|e^{\phi} w\|_{L^2(\R^n)} \leq C \Big(\|e^\phi |x|^2 Vw\|_{L^2(\R^n)}+\|e^{\phi}|x|^2 g\|_{L^2(\R^n)} + \max\{\tau, k\}\sum_{j=1}^n \|e^{\phi} |x| G^j \|_{L^2(\R^n)} \Big).
\end{align}
Considering $\tau\geq C\|V\|_{L^\infty(\Omega)}$, we can absorb the first term on the right hand side of \eqref{eq:proofcom} into the left hand side. 
Then, inserting the expressions for $w$, $g$, $G^j$ and $\phi$, we infer 
\begin{align*}
\begin{split}
\tau \||x|^{\tau} u\|_{L^2(B_2 \backslash \overline{B_{1+\delta}})} 
&\leq C\Big({\delta^{-2}}\||x|^{2+\tau} u\|_{L^2(B_{1+\sfrac{\delta}{2}}\backslash \overline{B_{1+\sfrac{\delta}{4}}})} +{\delta^{-1}} \||x|^{2+\tau} \nabla u\|_{L^2(B_{1+\sfrac{\delta}{2}}\backslash \overline{B_{1+\sfrac{\delta}{4}}})}\\
& \qquad\quad + \||x|^{2+\tau} f\|_{L^2(\Omega)} + (\delta^{-1}+\max\{\tau,k\})\sum_{j=1}^n\||x|^{1+\tau} F^j\|_{L^2(\Omega)}\Big).
\end{split}
\end{align*}
Hence, 
\begin{align*}
& \| u\|_{L^2(B_2 \backslash \overline{B_{1+\delta}})} \\
&\quad \leq C\delta^{-2} \left(\left(\frac{1+\sfrac{\delta}{2}}{1+\delta}\right)^{\tau+2} \| u\|_{H^1(B_{1+\sfrac{\delta}{2}}\backslash \overline{B_{1+\sfrac{\delta}{4}}})}  + 4^{\tau} k\Big(\| f\|_{L^2\Omega)} + \sum_{j=1}^n\| F^j\|_{L^2(\Omega)}\Big) \right)\\
&\quad \leq C{\delta^{-2}}\left(\left(\frac{1+\sfrac{\delta}{2}}{1+\delta}\right)^{\tau} \| u\|_{H^1(\Omega)}  + 4^{\tau} k\Big(\| f\|_{L^2(\Omega)} +\sum_{j=1}^n \| F^j\|_{L^2(\Omega)}\Big) \right)
\\ & \quad\leq C\delta^{-2} \left(\left(\frac{1+\sfrac{\delta}{2}}{1+\delta}\right)^{\tau} M + 4^{\tau} k\eta\right).
\end{align*}
Recalling that by assumption $k\eta \leq M$ and optimizing the right hand side by choosing $\tau= \tau_1 + \tau_0+ C\|V\|_{L^\infty(\Omega)}$ for $\tau_1>0$ such that 
\begin{align}
\label{eq:optimization1}
\left(\frac{1+\delta/2}{1+\delta}\right)^{\tau_1} M \sim 4^{\tau_1} k\eta.
\end{align}
This then implies the desired result with $\nu=1-\frac{\log 4}{\log 4+\log\frac{1+\delta}{1+\sfrac{\delta}{2}}}$.\\

\emph{Step 2: Proof of \eqref{eq:QUCPboundary_impr}.} We argue by making \eqref{eq:optimization1} more explicit. If $\delta\leq 1/2$ (which we can assume without loss of generality),  then 
\begin{align*}
\left(\frac{1+\delta/2}{1+\delta}\right)^{\tau} \leq \left(1-\frac{\delta}{3}\right)^{\tau}.
\end{align*}
 Hence, in the optimization argument we obtain 
\begin{align*}
\tau = \frac{1}{ \log\left(\frac{4}{1- \frac{\delta}{3} } \right) }\log\left( \frac{M}{k\eta} \right)+C\|V\|_{L^\infty(\Omega)}.
\end{align*}
As a consequence, 
\begin{align*}
\|u\|_{L^2(B_2 \backslash \overline{B_{1+\delta}})}
\leq C  {\delta^{-2}}(k\eta)^{\alpha} M^{1-\alpha},
\end{align*}
with $\alpha = 1-\frac{\log(2)}{\log(2)-\log(1-\frac{\delta}{3})}\geq c \delta$ and $C$ depending on $\|V\|_{L^\infty(\Omega)}$.

We combine this with an application of Hölder's inequality and Sobolev embedding close to the boundary:
\begin{align*}
\|u\|_{L^2(B_{1+\delta}\backslash \overline{B_{1}})}
\leq C \delta^{\frac{1}{n}} \|u\|_{L^{\frac{2n}{n-2}}(B_{1+\delta}\backslash \overline{B_{1}})} \leq C \delta^{\frac{1}{n}} \|u\|_{H^1(\Omega)}\leq C\delta^{\frac 1n}M.
\end{align*}
The combination of the two estimates then yields 
\begin{align*}
\|u\|_{L^2(B_2 \backslash \overline{B_1})} \leq C \left(\delta^{\frac{1}{n}}  + \delta^{-2}\Big(\frac{k\eta}{M}\Big)^{c\delta}\right)M.
\end{align*}
We now choose $\delta = c\log\left( \frac{M}{k\eta} \right)^{-\beta}$ for some $\beta\in(0,1)$. This  implies the claim with $\mu=\frac{2\beta}{n}$ (and a corresponding constant $C>0$ which depends on $\beta$).
\end{proof}

With Proposition ~\ref{prop:div_form_rhs} at our disposal, we next address the proof of Theorem ~\ref{prop:UCP_improved_ink}.

\begin{proof}[Proof of Theorem ~\ref{prop:UCP_improved_ink}]
We seek to reduce the problem with Cauchy data to the problem with a divergence form $H^{-1}$ right hand side. To this end, we argue by an extension argument. We note that by definition of the $H^{\half}(\partial B_2)$ norm there exists a function $v \in H^{1}(B_3 \backslash B_2)$ such that
\begin{align*}
\|v\|_{H^1(B_3 \backslash B_2)} \leq C \|u|_{\partial B_2}\|_{H^{\half}(\partial B_2)}.
\end{align*}
Let now $\chi \in C^{\infty}( B_3 \backslash \overline{B_2})$ be a smooth cut-off function  with $\chi|_{\p B_2}=1$ and $\chi|_{\p B_3}=0$.
We then define 
\begin{align}\label{eq:tildeu}
\tilde{u}:= \left\{
\begin{array}{ll}
u \;\;\mbox{ in } B_2 \backslash {B_1},\\
\chi v  \mbox{ in } B_3 \backslash \overline{B_2}.
\end{array}
\right.
\end{align}
This function then is an element of $H^1(B_3 \backslash \overline{B_1})$ with $\supp \tilde u\subset B_3\backslash B_1$. In addition, we claim that it is a weak solution to  
\begin{align}\label{eq:eqFjbound}
(\D + k^2 q+ V) \tilde{u} &= f+\sum_{j=1}^n\p_j F^j  \mbox{ in } B_3 \backslash {B_1},
\end{align}
where   $f, \, F^j \in L^2(B_3 \backslash {B_1})$ are functions supported in $B_3\backslash B_1$ and satisfying the bounds
\begin{align}\label{eq:estimatesFfbound}
 \|f\|_{L^2(B_3 \backslash B_1)}+\sum_{j=1}^n\|F^j\|_{L^2(B_3 \backslash B_1)} 
 \leq C k^2 \big(\|u|_{\partial B_2}\|_{H^{\half}(\partial B_2)} + \|\p_{\nu} u\|_{H^{-\half}(\partial B_{2})}\big).
\end{align}
Indeed, by the weak formulation of \eqref{eq:eqFjbound},  for $\varphi \in C_c^{\infty}(B_3 \backslash \overline{B_1})$,
\begin{align*}
\int\limits_{B_3 \backslash {B_1}} \big(\nabla \tilde{u} \cdot \nabla \varphi + k^2 q \tilde{u} \varphi + V \tilde{u} \varphi\big) dx
&= \int\limits_{B_2 \backslash {B_1}}\big( \nabla u \cdot \nabla \varphi + k^2 q u \varphi + V u \varphi \big)dx \\
& \quad + \int\limits_{B_3 \backslash \overline{B_2}} \big(\nabla ({\chi}v) \cdot \nabla \varphi + k^2 q {\chi}v\varphi + V {\chi}v \varphi \big)dx\\
& = - \int\limits_{\partial B_2} \p_{\nu} u \varphi dx
+ \int\limits_{B_3 \backslash \overline{B_2}} \big(\nabla ({\chi}v) \cdot \nabla \varphi + k^2 q {\chi}v\varphi + V {\chi}v \varphi\big) dx.
\end{align*}
Note that the mapping
\begin{align*}
\Psi: \varphi \mapsto \int\limits_{\partial B_2} \p_{\nu} u \varphi dx
\end{align*}
is bounded as an element in $H^{-\half}(\partial B_2)$ and also as an element  in $H^{-1}(B_2 \backslash \overline{B_1})$.
Indeed, for $\varphi\in H^1(B_2 \backslash \overline{B_1})$ we have
$$|\Psi(\varphi)|\leq \|\p_{\nu} u\|_{H^{-\half}(\p B_2)}\|\varphi\|_{H^{\half}(\p B_2)} \leq C\|\p_{\nu} u\|_{H^{-\half}(\p B_2)}\|\varphi\|_{H^{1}(B_2\backslash \overline{B_1})}.
$$
Therefore $\Psi\in (H^1(B_2 \backslash \overline{B_1}))^*\subset H^{-1}(B_2 \backslash \overline{B_1})$.
Then, it admits a representation $\Psi=g+\sum_{j=1}^n\p_j G^j$ with $g, G^j\in L^2(B_2 \backslash \overline{B_1}) $ and 
$$\|g\|_{L^2(B_2 \backslash {B_1})}+\sum_{j=1}^n\|G^j\|_{L^2(B_2 \backslash {B_1})}=\|\Psi\|_{H^{-1}(B_2 \backslash \overline{B_1})}\leq C\|\p_{\nu} u \|_{H^{-\half}(\p B_2)}.$$
As a consequence,  we obtain \eqref{eq:eqFjbound}  with 
\begin{align*}
f=\begin{cases}
-g &\mbox{ in }B_2\backslash B_1,\\
k^2 q \chi v+V\chi v &\mbox{ in }B_3\backslash \overline{B_2},
\end{cases},
\qquad 
F^j=\begin{cases}
-G^j &\mbox{ in }B_2\backslash B_1,\\
-\p_j(\chi v) &\mbox{ in }B_3\backslash \overline{B_2},
\end{cases}
\end{align*}
and \eqref{eq:estimatesFfbound} holds.
The result of Proposition ~\ref{prop:div_form_rhs} (rescaled to $B_3\backslash B_1$) is therefore applicable  with 
\begin{align}\label{eq:Metatilde}
\begin{split}
\tilde \eta&=Ck^2\eta\geq  \|f\|_{L^2(B_3 \backslash B_1)}+\sum_{j=1}^n\|F^j\|_{L^2(B_3 \backslash B_1)},\\
\tilde M&=CM\geq M+C\eta\geq \|\tilde u\|_{H^1(B_3\backslash B_1)}.
\end{split}
\end{align}
This yields the desired result.
\end{proof}

\subsection{Improved Runge approximation result}

\label{sec:Runge_improved}

This section contains the proof of Theorem ~\ref{thm:Rungeinterior_improv}.
We start by upgrading the  interior quantitative estimate from Theorem ~\ref{prop:UCP_improved_ink} similarly as in Proposition ~\ref{prop:QUCH1}.

\begin{prop}\label{prop:QUCH1_improv}
Let $V$ and $q$ be as in \textnormal{(\hyperref[assV]{i})}-\textnormal{(\hyperref[assq2]{ii'})} in $\Omega_2=B_2 \backslash \overline{B_{\half}}$ and let $\Omega_1=B_1 \backslash \overline{B_{\half}}$.
Let $u\in H^1(\Omega_2)$ be the unique solution to 
\begin{align*}
\begin{split}
\Delta u+k^2qu+Vu&=v\mathbb{1}_{\Omega_1} \quad  \mbox{ in } \Omega_2,\\
u&=0 \;\qquad\mbox{ on } \p\Omega_2,
\end{split}
\end{align*}
with $v\in L^2(\Omega_1)$ and $k\geq 1$ satisfying \textnormal{(\hyperref[assSpec]{a1})}.
Let $G=B_2\backslash \overline{B_{1+\delta}}$ for some $\delta\in (0,1)$. Then there exist  parameters $\nu_0\in(0,1)$, $s_0\in[3,n+1]$ and a constant $C>1$ (depending on $n, \delta, \|V\|_{L^\infty(\Omega_2)}, \kappa$ and $\|q\|_{C^1(\Omega_2)}$) such that
\begin{align}\label{eq:QUCPgradinterior_improv}
\|u\|_{H^1(G)}\leq  Ck^{s_0}\|\p_\nu u\|_{H^{-\half}(\p B_2)}^{\nu_0}\|v\|_{L^2(\Omega_1)}^{1-\nu_0}.
\end{align}
\end{prop}

\begin{proof}
We start by estimating $\|u\|_{H^1(\Omega_2)}$ in terms of $\|v\|_{L^2(\Omega_1)}$ as in Proposition ~\ref{prop:QUCH1}. By Lemma ~\ref{lem:apriori}
and \textnormal{(\hyperref[assSpec]{a1})}, there is $C>1$ such that
\begin{align*}
\|u\|_{H^1(\Omega_2)}\leq Ck^{n+1}\|v\|_{L^2(\Omega_1)}.
\end{align*}
Notice that then $u$ satisfies the assumptions in Theorem ~\ref{prop:UCP_improved_ink}, so \eqref{eq:QUCPinterior_impr} holds with 
\begin{align*}
M=Ck^{n+1}\|v\|_{L^2(\Omega_1)},\qquad \eta=\|\p_\nu u\|_{H^{-\half}(\p B_2)}.
\end{align*}

Let us now show that the bound \eqref{eq:QUCPinterior_improv} can be upgraded to an estimate for the $H^1$ norm. 
This is inherited from \eqref{eq:QUCPinterior_impr}. Indeed, we argue as in Step 1 of the proof of Proposition ~\ref{prop:div_form_rhs}, but now including into the left hand side of  \eqref{eq:proofcom} the gradient term $\|e^\phi|x|\nabla w\|_{L^2(\R^n)}$ coming from Proposition ~\ref{prop:Carl_eigenval}.
Therefore,  if $k\tilde \eta\leq\tilde M$,
\begin{align*}
\|\tilde u\|_{H^1(G)}\leq C\left(\frac{k\tilde\eta}{\tilde M}\right)^{\nu}\tilde M,
\end{align*}
where $\nu$ and $C$ depend  in particular on $\delta$. Here $\tilde u$ is given by \eqref{eq:tildeu} and $\tilde M$ and $\tilde \eta$ are connected with $M$ and $\eta$ according to \eqref{eq:Metatilde}. 
Following the proof of Theorem ~\ref{prop:UCP_improved_ink}, we then  obtain 
\begin{align*}
\|u\|_{H^1(G)}
\leq C\left(\frac{k^3\eta}{M}\right)^{\nu_0}M,
\end{align*}
if $k^3\eta\leq M$.
Otherwise, if $k^3\eta\geq M$, the estimate is immediate.
Therefore the final bound \eqref{eq:QUCPgradinterior_improv} holds with $s_0=3\nu_0+(n+1)(1-\nu_0)$. 
\end{proof}

With Proposition \ref{prop:QUCH1_improv} we deduce the proof of Theorem \ref{thm:Rungeinterior_improv} similarly as in the analogous non-convex settings:

\begin{proof}[Proof of Theorem ~\ref{thm:Rungeinterior_improv}]
The proof follows  the proof of Theorem ~\ref{thm:Rungeinterior} in Section ~\ref{sec:Runge} with $\Gamma=\p B_2$ in order to construct $u_\alpha, v_\alpha$ and $w_\alpha$. 
The difference appears at the time of estimating $\|w_\alpha\|$. Applying the improved estimate \eqref{eq:QUCPgradinterior_improv} instead of \eqref{eq:QUCPgradinterior},
we obtain
\begin{align*}
\|u_\alpha-u\|_{L^2(\Omega_1)}\leq k^{s_0} \left(\frac{\|\p_\nu w_\alpha\|_{H^{-\half}(\p B_2)}}{\|v_\alpha\|_{L^2(\Omega_1)}}\right)^{\nu_0}\|\tilde v\|_{H^1(\tilde\Omega_1)}
\leq Ck^{s_0}\alpha^{\nu_0}\|\tilde v\|_{H^1(\tilde\Omega_1)}.
\end{align*}
Choosing $\alpha$ such that $Ck^{s_0}\alpha^{\nu_0}=\epsilon$, we finally deduce
\begin{align*}
\|u_\alpha\|_{H^{\half}(\p B_2)}\leq \frac{1}{\alpha} \|\tilde v\|_{L^2(\Omega_1)}=Ck^{\frac {s_0} {\nu_0}}\epsilon^{-\frac 1 {\nu_0}}\|\tilde v\|_{L^2(\Omega_1)}.
\end{align*}
\end{proof}

\section*{Acknowledgements}
A.R. was supported by the Deutsche Forschungsgemeinschaft (DFG, German ResearchFoundation) under Germany’s Excellence Strategy EXC-2181/1 - 390900948 (the Heidelberg STRUCTURES Cluster of Excellence).
W.Z. was supported by the European Research Council (ERC) under the Grant Agreement No 801867.

\bibliographystyle{alpha}
\bibliography{citationsHT}

\end{document}